%% file: main.tex
\documentclass[11pt,letterpaper]{article}
\usepackage[utf8]{inputenc}
\usepackage{amsfonts}
\usepackage{amssymb}
\usepackage{fullpage,tikz}
\input{myheader.tex}

\usepackage{thm-restate}

\DeclareMathOperator{\Cay}{Cay}

\title{Sparse graph counting and Kelley--Meka bounds for binary systems}

\author{
 Yuval Filmus\thanks{The Henry \& Marilyn Taub Faculty of Computer Science and Faculty of Mathematics, Technion. Email: \texttt{yuvalfi@cs.technion.ac.il}
 }
 \and
	Hamed Hatami\thanks{School of Computer Science, McGill University. Email: \texttt{hatami@cs.mcgill.ca} 
 }
    \and Kaave Hosseini\thanks{Department of Computer Science, University of Rochester. Email: \texttt{kaave.hosseini@rochester.edu}
    }	
    \and 
     Esty Kelman\thanks{Department of Computer Science and Department of Computing \& Data Sciences, Boston University and CSAIL, Massachusetts Institute of Technology. Email: \texttt{ekelman@mit.edu}. Supported in part by NSF TRIPODS program (award DMS-2022448).
     }
}

\date{\today}

\begin{document}
 
\maketitle
\begin{abstract}
In a recent breakthrough, Kelley and Meka (FOCS 2023) obtained a strong upper bound on the density of sets of integers without non-trivial three-term arithmetic progressions. In this work, we extend their result, establishing similar bounds for all linear patterns defined by binary systems of linear forms, where  ``binary'' indicates that every linear form depends on exactly two variables. Prior to our work, no strong bounds were known for such systems even in the finite field model setting.

 A key ingredient in our proof is a  \emph{graph counting lemma}. The classical graph counting lemma, developed by Thomason (Random Graphs 1985) and Chung, Graham, and Wilson (Combinatorica 1989), is a fundamental tool in combinatorics. For a fixed graph $H$, it states that the number of copies of  $H$ in a pseudorandom graph $G$ is similar to the number of copies of $H$ in a purely random graph with the same edge density as $G$. However, this lemma is only non-trivial when $G$ is a dense graph. 
In this work, we prove a graph counting lemma that is also effective when $G$ is sparse. Moreover, our lemma is well-suited for density increment arguments in additive number theory. 

As an immediate application, we obtain a strong bound for the Tur\'an problem in abelian Cayley sum graphs: let $\Gamma$ be a finite abelian group with odd order. If a Cayley sum graph on  $\Gamma$ does not contain any $r$-clique as a subgraph, it must have at most $2^{-\Omega_r(\log^{1/16}|\Gamma|)}\cdot |\Gamma|^2$ edges.

These results hinge on the technology developed by Kelley and Meka and the follow-up work by Kelley, Lovett, and Meka (STOC 2024).

\end{abstract}

\thispagestyle{empty} 
\newpage

\setcounter{page}{1}

\section{Introduction}

The \emph{graph counting lemma}, due to Thomason~\cite{MR0930498} and Chung, Graham, and Wilson~\cite{CGW89}, is a fundamental result in graph theory, with many applications in extremal combinatorics, additive number theory, discrete geometry, and theoretical computer science. Informally, it states that the number of embeddings of a fixed graph $H$ into a random-looking (pseudorandom) graph $G$ is similar to the number of embeddings of $H$ into a genuine random graph with the same edge density as $G$.

For certain problems, particularly in additive number theory, it is possible to use the counting lemma to obtain a density increment argument. Roughly speaking, if the number of embeddings of $H$ into $G$  deviates from that of a random graph, then by the counting lemma, $G$ is not pseudorandom and, therefore, contains some structure. In certain cases, it is possible to exploit this structure to restrict the problem to a significantly denser subproblem and repeat the argument. As the density cannot exceed one, this process must stop after a bounded number of iterations.

Many of the strongest known bounds in additive number theory and extremal combinatorics rely on the density increment approach. This method often yields much better bounds compared to other common techniques, such as the regularity method.

The classical graph counting lemma has a shortcoming: it is only effective for dense graphs.  Indeed, while the counting lemma and the regularity method provide an elegant theory for understanding the subgraph densities in dense graphs~\cite{MR3012035}, obtaining an analogous theory for the sparse setting has proven challenging. Over the past few decades, considerable effort has been devoted towards developing sparse analogs of the counting and regularity lemmas (see for example~\cite{MR1661982, MR2187740, chung2002sparse, kohayakawa2003regular,  kohayakawa2004embedding, kohayakawa2010triangle, conlon2014extremal, allen2020regularity}).

The conventional strategy for extending these results to sparse graphs is to introduce an additional requirement: that the graph $G$, even though sparse, is a relatively dense subgraph of a much more pseudorandom host graph $\Gamma$. For example,  in an important paper, Conlon, Fox, and Zhao~\cite{conlon2014extremal}  proved an effective sparse counting lemma under this additional requirement. Their result generalizes an earlier theorem of Kohayakawa, R{\"o}dl, Schacht, and Skokan~\cite{kohayakawa2010triangle} that only applied to counting triangles. 
However, the additional requirement that $G$ must live in a much more pseudorandom graph $\Gamma$ makes such sparse counting lemmas unsuitable for certain applications of the original counting lemma.

In this paper, we prove a counting lemma for sparse graphs (\cref{thm:regular_count}) that circumvents this extra assumption on $G$. Our theorem states that if the number of embeddings of $H$ into $G$  deviates from what is expected, then either $G$ contains many vertices of small degree, or there exist large subsets $S, T \subseteq V(G)$ such that the density of the induced bipartite graph between them is notably \emph{higher} than the density of $G$. Due to its guarantee of increased density, our counting lemma is tailored for density increment arguments.

We use our new sparse counting lemma alongside various tools from additive number theory, such as the Bohr set machinery, almost periodicity, and dependent random choice, to extend Kelley and Meka's strong bound on $3$-progression-free sets to a broad and important class of more general linear patterns. We will postpone the detailed discussion of this result to \Cref{sec:intro_kelleyMeka}.

\subsection{The sparse graph counting lemma}  
We identify undirected graphs as symmetric functions $A\colon \cX \times \cX \to \{0,1\}$ through the correspondence with the characteristic function of the edges. The \emph{density of $A$} is
\[\Ex[A] \defeq \Ex_{\substack{\bfx_1,\bfx_2 \in \cX}}[A(\bfx_1,\bfx_2)],\] 
where $\bfx_1,\bfx_2 \in \cX$ means that $\bfx_1$ and $\bfx_2$ are chosen uniformly and independently from $\cX$.

Let $H = ([k], E_H)$ be a fixed undirected graph with vertex set $[k] \defeq \{1,\ldots,k\}$ and $m=|E_H|$ edges. We define the \emph{density of $H$ in $A$} as  
 \[t_H(A) \defeq \Ex_{\substack{\bfx_1,\ldots,\bfx_k \in \cX}}  \left[\prod_{\{i,j\} \in E_H} {
A}(\bfx_i,\bfx_j)\right].\]

When $A$ is the Erd\H{o}s--R\'{e}nyi random graph with parameter $\alpha$, we have $t_H(A) \approx \alpha^m$, with high probability. For an arbitrary $A$, the graph counting lemma of \cite{MR0930498,CGW89} guarantees $t_H(A) \approx \alpha^m$ if $A$  is pseudorandom in a certain structural sense. We state this result in the contrapositive form.

\begin{theorem}[{Graph Counting Lemma~\cite{MR0930498,CGW89}}]\label{theorem:CGW}
For every fixed graph $H$ with $m$ edges and $\eps \in (0,1)$, there exists $\delta=\delta(\eps,m)>0$ such that the following is true. If $A$ is a graph on $\cX$ with density $\alpha$ that satisfies 
\[
    \left|t_H(A)- \alpha^m \right| \ge \eps,
\]
there must exist $S,T \subseteq \cX$ with densities $\frac{|S|}{|\cX|}, \frac{|T|}{|\cX|} \ge  \Omega_{m,\eps}(1)$
such that   
\[
\left|\Ex_{\substack{(\bfx,\bfy) \in S\times T}}[A(\bfx,\bfy)] -\alpha\right|\geq \delta. 
\]
\end{theorem}
In other words, if $t_H(A)$ is significantly different from what is expected from a random graph of density $\alpha$, then on some large combinatorial rectangle $S \times T$,  the density of $A$ significantly deviates from $\alpha$.

In the sparse setting, where $\alpha = o(1)$, we have $|t_H(A)-\alpha^m|=o(1)$, and therefore,  to obtain a useful counting lemma for sparse graphs, we need to refine the assumption of the counting lemma to 
\begin{equation}
\label{eq:stronger_assumption}
    \left|t_H(A)- \alpha^m \right| \ge \eps \alpha^m.
\end{equation}
The following well-known example shows that \eqref{eq:stronger_assumption} does not necessarily imply the existence of $S,T \subseteq \cX$ with density $\Omega_{m,\eps}(1)$  that satisfy
\begin{equation}
\label{eq:stronger_conclusion}
\left|\Ex_{\substack{(\bfx,\bfy) \sim S\times T}}[A(\bfx,\bfy)] -\alpha\right|\geq \Omega(\alpha).
\end{equation}
 
\begin{example}\label{example:trianglesparse}
Let $H$ be the triangle graph, and let $A$ be the Erd\H{o}s--R\'{e}nyi random graph on $n$ vertices with parameter $\alpha = n^{-2/3}$. With high probability, we have  $t_H(A) \approx \alpha^3 = n^{-2}$, and therefore, there are at most $O(n)$ triangles in $A$. By removing at most $O(n)$ edges from $A$, we can obtain a triangle-free $A'$.  Since $A'$ is triangle-free and its density is very close to $\alpha$, it satisfies the assumption \eqref{eq:stronger_assumption}.   Moreover, 
\begin{equation*}
\left|\Ex_{\substack{(\bfx,\bfy) \sim S\times T}}[A'(\bfx,\bfy)] -\alpha\right| =  o(\alpha) 
\end{equation*}
for all subsets $S,T\subseteq [n]$ of density at least, say,  $n^{-1/7}$. Therefore, $A'$ violates the desired conclusion \eqref{eq:stronger_conclusion}.
\end{example}
To circumvent such examples, Conlon, Fox and Zhao~\cite{conlon2014extremal} assume further that $A$ is a relatively dense subgraph of a host graph $\Gamma$ that satisfies much stronger pseudo-randomness conditions. Their counting lemma states that under this extra requirement, \eqref{eq:stronger_assumption} indeed implies \eqref{eq:stronger_conclusion} for some $S,T \subseteq \cX$ with density $\Omega_{m,\eps}(1)$.
We prove a counting lemma that does not require such an extra pseudorandomness condition on $A$.  

\begin{theorem}
\label{theorem:low_disc_count} 
Let $H$ be a graph with $m$ edges, and let $\eps \in (0,1)$. For every graph $A$ with vertex set $\cX$ and density $\alpha$, the following holds.
\begin{itemize}
\item[(i)] If $t_H(A) \ge (1+\eps)^m \alpha^m$, there exist $S,T \subseteq \cX$ with  $\frac{|S||T|}{|\cX|^2}=\Omega(\eps\alpha^{m+1})$ such that   
\[\Ex_{\substack{(\bfx,\bfy) \sim S\times T}}[A(\bfx,\bfy)] \ge \left(1+\frac{\eps}{2}\right) \alpha. \]
\item[(ii)] If $t_H(A) \le (1-\eps)^m \alpha^m$, there exist $S,T \subseteq \cX$ with  $\frac{|S||T|}{|\cX|^2}=\Omega(\eps(1-\eps)^{m-1} \alpha^m)$ such that   
\[\Ex_{\substack{(\bfx,\bfy) \sim S\times T}}[A(\bfx,\bfy)] \le \left(1-\frac{\eps}{2}\right) \alpha. \]
\end{itemize}
\end{theorem}
\Cref{theorem:low_disc_count} is a special case of a slightly more general result (\Cref{thm:low_disc_graph_count}) that we state and prove later.  \Cref{theorem:low_disc_count} is not our main counting lemma, and as explained later, it is not useful for strong density increment applications. However, it can be viewed as an extension of \Cref{theorem:CGW} to the sparse setting and provides a nice complement to \Cref{example:trianglesparse}. 

Kelley, Lovett, Meka~\cite{KLM} proved the first part of \Cref{theorem:low_disc_count} using a smooth analog of the dependent random choice method in the special case when $H$ is a complete bipartite graph. In \Cref{theorem:low_disc_count}, we generalize their result to all graphs using a similar proof.

\paragraph{A density increment counting lemma.} \Cref{theorem:low_disc_count} is not useful for obtaining density increment arguments since, in case (ii), it only provides a combinatorial rectangle with \emph{decreased} density.

Let us consider the original counting lemma (\cref{theorem:CGW}). By tracking the parameters in the proof of this theorem, one can easily show that \eqref{eq:stronger_assumption} implies the existence of $S,T \subseteq \cX$ with densities $\Omega_{m,\eps}(1)$ such that
\[
\Ex_{\substack{(\bfx,\bfy) \sim S\times T}}[A(\bfx,\bfy)] \geq \alpha+\Omega(\alpha^m
).
\]

While this result provides a density increase,   the magnitude of this increase is insignificant when  $\alpha$ is small. In many applications, having an increase of $\Omega(\alpha)$ is crucial. For example,  a density increment that multiplies the density by a factor of $1+\Omega(1)$ can be repeated at most $O(\log(2/\alpha))$ times. In contrast, a weaker density increment by a factor of $1+ \Omega(\alpha^{c})$ for $c>0$ needs to be repeated at least $\Omega(\alpha^{-c} \log(2/\alpha))$ times before it reaches a contradiction. Since each iteration decreases the size of the domain by a constant factor, this large number of iterations would only be possible if $\alpha \ge \Omega(1/\log |\cX|)$.

Our main graph counting lemma provides a significant density increment when $A$  does not contain many vertices with small degrees. 

\begin{theorem}[Main graph counting lemma]\label{thm:regular_count}
Let $H$ be a graph with $m$ edges, and let $\eps \in (0,1)$. There exists $\delta=\delta(\eps,m)>0$ such that for every graph $A$ of density $\alpha$ on a vertex set $\cX$, the following holds.  If 
\begin{equation*} 
    \left|t_H(A)- \alpha^m \right| \ge \eps \alpha^m,
\end{equation*}
then, either $A$ has $2^{-O_{m,\eps}(\log(2/\alpha))}|\cX|$ vertices with degrees at most $(1-\delta)\alpha |\cX|$, or there exist $S,T \subseteq \cX$ of densities at least $2^{-O_{m,\eps}(\log^2(2/\alpha))}$ satisfying
\[
\Ex_{\substack{(\bfx,\bfy) \in S\times T}}[A(\bfx,\bfy)] \geq (1+\delta) \alpha.
\]
\end{theorem}

\Cref{thm:regular_count} states that if $A$ is a graph with $|t_H(A)-\alpha^m| \ge \eps \alpha^m$, then either it has many vertices with small degrees, or the density of $A$ on some large rectangle is greater than $(1+\delta)\alpha$. The significance of this theorem lies in the fact that even when  $t_H(A)$ is \emph{smaller} than, say, $\alpha^m/2$, it provides a rectangle with \emph{increased} density. This feature makes the theorem suitable for density increment arguments, especially in additive number theory, where one naturally deals with regular graphs.

\cref{thm:regular_count} is a special case of a more general result (\cref{thm:main_graph_count}) which we state and prove later.

\subsection{Kelley--Meka bounds for binary systems of linear forms}\label{sec:intro_kelleyMeka}

 To study the number of occurrences of a linear pattern (e.g., $3$-progressions) in a subset of the interval $\{1,\ldots,M\}$, it suffices to embed $\{1,\ldots,M\}$ in $\mathbb{Z}_N$ for a prime $N=O(M)$ chosen sufficiently large to avoid wraparound.  Consequently, rather than working with the interval, one can focus on subsets of finite abelian groups $G$. The two important cases are $G=\mathbb{Z}_N$, where $N$ is a large prime and asymptotics are as $N$ tends to infinity, and $G=\mathbb{F}_q^n$, where $q$ is a fixed prime and asymptotics are as $n$ tends to infinity. The former is important for applications in number theory, and the latter is the setting that is most relevant to applications in combinatorics and computer science.

In 1953, Roth~\cite{roth1953certain} proved that a set of integers  $A \subseteq \Z_N$ without non-trivial 3-progressions must have density $O(1/\log \log N)$.  The quest to determine the optimal bound in Roth's theorem stands as a central problem in additive number theory, which has been studied extensively over the past seven decades (notably~\cite{MR0889362,MR1100788,MR1726234,MR2811612,bloom2021breaking,KM23}). For a long time, $1/\log N$ was perceived as a barrier for all the available techniques.  In~\cite{bloom2021breaking}, Bloom and Sisask overcame this barrier and established the bound $O(1/\log^{1+\eps} N)$ for some small $\eps>0$. More recently, in a remarkable breakthrough, Kelley and Meka~\cite{KM23} took a huge leap forward and improved the bound to  $2^{-\Omega(\log^{1/12}N)}$. Subsequently, this bound was refined to  $2^{-\Omega(\log^{1/9}N)}$ by Bloom and Sisask~\cite{BS23improvment}.  On the other hand, Behrend's classical construction~\cite{MR0018694} shows that there are sets with density   $2^{-O(\log^{1/2}N)}$ that are free of non-trivial $3$-progressions.

Similarly to all previous approaches, \cite {KM23} uses a density increment argument.  Note that the density of 3-progressions in a set $A \subseteq \mathbb{Z}_N$ is given by 
\begin{equation}\label{equation:3AP_def}
t_{\textrm{3AP}}(A)\coloneqq\Ex_{\bfx,\bfy\in \Z_N} \left[A(\bfx) A\left(\bfx+\bfy\right) A(\bfx+2\bfy)\right], 
\end{equation}
where we identified $A$ with its indicator function. 
If $A$ is a random set of density $\alpha$, then we expect \(t_{\textrm{3AP}}(A)\approx \alpha^3\).
Kelley and Meka~\cite{KM23} show that if a set $A$ with density $\alpha$ satisfies 
\begin{equation*}
|t_{\textrm{3AP}}(A)-\alpha^3| \ge  \frac{\alpha^3}{2}, 
\end{equation*}
then restricting to some large low-rank Bohr set  $B$ increases the relative density of $A$ by a multiplicative factor of $1+\Omega(1)$. Repeating this argument iteratively for at most $O(\log(2/\alpha))$ times shows that if $A$ does not contain any non-trivial $3$-progressions, then $\alpha \leq 2^{-\Omega(\log^{1/12}{N})}$.

Most linear patterns, such as $4$-progressions, are fundamentally beyond the scope of such a  Bohr set counting lemma: there are sets $A$ with density $\alpha = \Omega(1)$ that satisfy $|t_{\textrm{4AP}}(A)-\alpha^4|= \Omega(1)$, but the relative density of $A$ in every large low-rank Bohr set remains extremely close to $\alpha$, see \cite{MR1844079}.  

\begin{remark}
In 1975, Szemer{\'e}di~\cite{szemeredi1975sets} extended Roth's $o(1)$ bound using a completely new approach to arithmetic progressions of arbitrary length. However, in contrast to Roth's Fourier-analytic proof, all the various known proofs of Szemer{\'e}di's theorem give much weaker bounds. Indeed, it was considered a major breakthrough when Gowers~\cite{MR1844079} proved an upper bound of $1/(\log \log N)^{2^{-2^{k+9}}}$ on the density of sets of integers without $k$-term arithmetic progressions.  Very recently, Leng, Sah, and Sawhney posted a preprint~\cite{leng2024improved}  improving Gowers' bound to an impressive bound of $O(2^{-(\log\log N)^{c_k}})$, where $c_k>0$ is a constant depending on $k$. 
\end{remark}

Nonetheless, there is an important class of linear patterns that lie somewhere between $3$-progressions and $4$-progressions in terms of their ``complexity''. These linear patterns, defined by \emph{binary systems of linear forms}, stem from the graph-theoretic approach of Ruzsa and Szemer\'edi~\cite{MR0519318} to Roth's theorem. 
Ruzsa and Szemer\'edi discovered a versatile proof of Roth's theorem by establishing and applying the so-called triangle removal lemma, which is a consequence of Szemer\'edi's regularity lemma. Their proof uses the expression 
\begin{equation} 
\label{eq:RuzsaSzemeredi}
t_{\textrm{3AP}}(A)=\Ex_{\bfx,\bfy,\bfz\in \Z_N} \left[A(2\bfx-2\bfy) A\left(\bfx-\bfz\right) A(2\bfy-2\bfz)\right]  
\end{equation}
to represent the density of  3-progressions in $A$ as the density of triangles in a tripartite graph constructed from $A$, and then applies these graph theoretic tools. 

In~\cite{MR0932119}, Erd\H{o}s, Frankl, and R\"{o}dl proved the graph removal lemma, which generalizes the triangle removal lemma to all graphs. This generalization extends the graph-theoretic approach to a large class of linear patterns, namely those defined by \emph{binary systems of linear forms}. Essentially, these linear patterns have a graph-like structure, similar to how $3$-progressions have a triangle-like structure. We will formally define these systems in \Cref{{subsection:notationandterm}}, but for now, it suffices to mention that the term \emph{binary} indicates that each linear form is supported on exactly two variables.  Note that $(2x-2y,x-z,2y-2z)$, which captures $3$-progressions, satisfies this property. As another notable example of a binary system, consider
\begin{equation}\label{eq:K4}
  t_{K_4}(A) \coloneqq\Ex_{\bfx,\bfy,\bfz,\bfw\in \Z_N} \left[A(\bfx+\bfy)A(\bfx+\bfz)A(\bfx+\bfw)A(\bfy+\bfz)A(\bfy+\bfw)A(\bfz+\bfw)\right],  
\end{equation}
which corresponds to the density of $K_4$ (the complete graph on $4$ vertices) in the Cayley sum graph $\Cay(\mathbb{Z}_N,A)$, whose edges are $\{x,y\}$ for $x+y \in A$.

The graph-theoretic approach relies on Szemer\'edi's regularity lemma, and therefore results in an extremely weak bound compared to Roth's original bound. However, one can use a common technique in additive number theory and extremal combinatorics, based on a few careful applications of the Cauchy--Schwarz inequality, to extend Roth's original Fourier-analytic proof and upper bound to all binary systems.

Unfortunately, the aforementioned Cauchy--Schwarz approach incurs too much loss to extend Kelley and Meka's upper bound to such a broad class of linear patterns. 

Previous techniques developed by Green and Tao (in their work on $4$-progressions \cite{greenT2009new}) and by Dousse \cite{dousse2013generalisation} extend to binary systems\footnote{and more generally to systems of complexity $1$.}, and for such a broad class of systems, the strongest known bound is $1/\log^c N$ (where $c$ is a small constant depending on the system), established by Shao \cite{Shao14finding}.

In fact, a priori, it is unclear whether Kelley and Meka's techniques could be applied to, for instance, $t_{K_4}(A)$. Let us briefly review the key steps of their proof:
\begin{enumerate}
\item Their proof starts from the assumption $|t_{\textrm{3AP}}(A)-\alpha^3| \ge \eps \alpha^3$, and notes that 
\[t_{\textrm{3AP}}(A)-\alpha^3 = \inp{(A-\alpha)*(A-\alpha)}{C},  \] 
where $C=\set{2a \mid a \in A}$. 
\item (H\"older step): Since the density of $C$ is $\alpha$, one can use H\"older's inequality to show that the right-hand side is at most $O(\alpha \norm{(A-\alpha)*(A-\alpha)}_p)$ for $p \approx \log(2/\alpha)$.  

\item (Positivity/Significant Increment): Then, in one of the key steps of the proof, the proof uses the fact that the Fourier coefficients of the convolution $(A-\alpha)*(A-\alpha)$ are positive to deduce from $\norm{(A-\alpha)*(A-\alpha)}_p=\Omega(\alpha^2)$ that 
\[\norm{A*A}_p = \left(1+\Omega(1)\right) \alpha^2.\]
\item (Sifting/Dependent random choice): This key step involves the application of the dependent random choice technique to the lower bound on $\norm{A*A}_p$.
\item (Almost periodicity): At this point, one can combine the result of the previous step with the available almost periodicity result of~\cite{SandersBogo} to achieve the desired density increment.
\end{enumerate}

Note that several steps in the above proof rely on the properties of convolution, and while one can express $t_{\textrm{3AP}}(A)$ by a simple formula involving convolution, this is not true for other more complex binary systems such as $t_{K_4}(A)$. \textit{Is it still true that \(|t_{K_4}(A)-\alpha^6| = \Omega(\alpha^6)\) implies a significant density increment in a large low-rank Bohr set?}
If \(|t_{K_4}(A)-\alpha^6| = \Omega(\alpha^6)\), then our counting lemma (\cref{thm:regular_count}) provides two sets $S$ and $T$ with densities $2^{-O(\log^2 (2/\alpha))}$ such that
\begin{equation}
\label{eq:spreadness_example}
(1+\Omega(1))\alpha_S \alpha_T \alpha \le \Ex_{\bfx,\bfy \in \mathbb{Z}_N} \left[S(\bfx)T(\bfy)A(\bfx+\bfy)\right],
\end{equation}
where $\alpha_S$ and $\alpha_T$ are the densities of $S$ and $T$, respectively. 

This is promising, as $\Ex\left[ S(\bfx)T(\bfy)A(\bfx+\bfy)\right]$ is a much simpler expression than $t_{K_4}(A)$, and moreover, we have shown in \eqref{eq:spreadness_example} that it is significantly larger than its ``expected'' value 
$\alpha_S \alpha_T \alpha$.   
By establishing a sifting step for general finite abelian groups that is applicable to \eqref{eq:spreadness_example}, we prove the following theorem.

\begin{restatable}[Kelley--Meka-type bounds over finite abelian groups]{thm}{GeneralizedKelleyMekaAbelian}
\label{thm:L-freeAbelian}
Let $\cL$ be a binary system of linear forms over a finite abelian group $G$, and suppose $|G|$ is coprime with all the coefficients in the linear forms. If $A \subseteq G$  is free of translations of non-degenerate instances of $\cL$, then
\begin{itemize}
    \item[(i)]  \label{itm:L-freeAbelian-general}
    \hfill
    \(\displaystyle |A| \le |G|\cdot 2^{-\Omega_{\cL}(\log^{1/16}|G|)}. \)
    \hfill \mbox{}
    \item[(ii)]\label{itm:KM_2_degenarate} If, in addition, the underlying graph of $\cL$ is $2$-degenerate, then
    \[|A| \le |G|\cdot 2^{-\Omega_{\cL}(\log^{1/9}|G|)}. \]
\end{itemize}

\end{restatable}
Here, an undirected graph $H$ is called \emph{$k$-degenerate} if every subgraph of $H$ has a vertex of degree at most $k$. 
\begin{remark}
It is worth noting that even in the case of the group $G= \F_q^n$, despite the success of the polynomial method in proving strong bounds for Roth's theorem~\cite{MR3583357,MR3583358}, the bound of \cref{thm:L-freeAbelian} is new. 
In 2017, \cite{MR3583357,MR3583358} used the polynomial method to prove that any subset of $G=\F_q^n$ without $3$-progressions must be of size at most $c^n$, where $c=c(q)<q$ is a fixed constant.  This bound is stronger than Kelley and Meka's bound, and indeed, the significance of Kelley and Meka's result lies in the integer case, where there is no known analog of the polynomial method. 
However, even in the setting of $\F_q^n$, the polynomial method appears to be limited in dealing with general binary systems. In a recent article~\cite{MR4684863}, Gijswijt extended the polynomial method to a larger class of systems of linear forms. However, even this class does not include the binary systems whose underlying graphs are dense graphs such as $K_4$ or larger cliques. For cases such as \eqref{eq:K4},  the strongest known previous bounds were of the form $q^{n-\Omega_q(\log n)}$. \Cref{thm:L-freeAbelian} improves these bounds to $q^{n-\Omega_q(n^{1/16})}$. 
\end{remark} 
  
\begin{remark}
In the case of 3-progressions, currently the strongest known bound for $G=\mathbb{Z}_N$ is $|A| \le |G|\cdot 2^{-\Omega(\log^{1/9}|G|)}$, due to Bloom and Sisask~\cite{BS23improvment}, which is stronger than what we obtain for general binary systems in \cref{thm:L-freeAbelian}~(i). To briefly explain this disparity, note that we need to apply our counting lemma to obtain  \eqref{eq:spreadness_example}, and the two resulting sets, $S$ and $T$, are of densities $2^{-O(\log^2 (2/\alpha))}$, which is small compared to $\alpha$. In contrast, in the case of 3-progressions, one can directly work with $\|A*A\|_p$, which leads to a stronger bound.  Nevertheless, in \Cref{itm:KM_2_degenarate}~(ii) we show that if the underlying graph of the binary system is $2$-degenerate (e.g., the triangle graph for 3-progressions), we can obtain the same upper bound $|A| \le |G|\cdot 2^{-\Omega(\log^{1/9}|G|)}$.  
\end{remark}

\subsection{Tur\'an's theorem for Cayley sum graphs over abelian groups}
Here, we present a quick application of \Cref{thm:L-freeAbelian} to Tur\'an's theorem, a classical result in extremal graph theory regarding the largest edge density of a graph that does not contain any copies of $K_{r}$ (the complete graph on $r$ vertices), for some given $r$. Tur\'an proved that this density is at most $1-\frac{1}{r-1}+o(1)$, which is tight, as witnessed by a complete $(r-1)$-partite graph where the parts have an almost equal number of vertices.

Recall that given a finite abelian group $\Gamma$ and a set $A\subseteq \Gamma$, the Cayley sum graph $\Cay(\Gamma,A)$ is the graph with vertex set $\Gamma$ and edge set $\{\{x,y\} \mid x+y \in A\}$.   \Cref{thm:L-freeAbelian}~(i) implies a strong upper bound on the edge density of $G$ in Tur\'an's theorem if $ G$ is further assumed to be a Cayley sum graph of a finite abelian group $\Gamma$ with $2\not| \ |\Gamma|$.

\begin{corollary}\label{cor:cayleyTuran}
Let $r \ge 2$ be a fixed positive integer. Let $\Gamma$ be a finite abelian group such that $2\not| \ |\Gamma|$, and let $A\subseteq \Gamma$. If $\Cay(\Gamma,A)$ does not contain any copies of $K_{r}$, then its edge density is at most $2^{-\Omega_r(\log^{1/16}|\Gamma|)}$.
\end{corollary}
\begin{proof}
    The system of linear forms associated with $r$-clique is $\cL = \{x_i+x_j : i\neq j, i,j\in [r]\}$.
    Next, we show that if $2\not | \ |\Gamma|$, then $\cL$ is translation invariant (see \Cref{definition:ti} on page~\pageref{definition:ti}), and therefore, we can apply \Cref{thm:L-freeAbelian} to conclude the corollary.
    We need to show that every shift of an $r$-clique is also an $r$-clique. This follows from the identity $x_i+x_j+c = (x_i+ c/2)+(x_j+c/2)$ for all $i\neq j$, which is valid if $2\not| \ |\Gamma|$.
\end{proof}

Note that the requirement $2\not| \ |\Gamma|$ is crucial. For example, let $\Gamma = \mathbb{Z}_{2N}$ and let $A$ be the set of all odd numbers in $\mathbb{Z}_{2N}$, namely, $A = \{2i+1 \bmod 2N: i\in [N]\}\subseteq \Gamma$. The Cayley sum graph $\Cay(\Gamma,A)$ does not contain any triangles because for any three elements $x,y,z\in \mathbb{Z}_{2N}$, at least one of $x+y,x+z,y+z$ is even, hence not in $A$. However, $\Cay(\Gamma, A)$ has edge density $\approx\frac{1}{2}$.

\paragraph{Acknowledgment} 
We wish to thank Zander Kelley for bringing \cite[Theorem 4.10]{KM23} to our attention, which proved invaluable in extending \Cref{thm:L-freeAbelian} from the 2-degenerate case to all binary systems. We wish to thank Tsun-Ming Cheung and Shachar Lovett for their valuable feedback on an earlier draft of this paper. K.H. would like to thank Anup Rao for valuable conversations. We wish to thank Dion Gijswijt for bringing \cite{MR4684863} to our attention. We also thank Seth Pettie for pointing out a mistake in a previous version of the paper.

This project has received funding from the European Union's Horizon 2020 research and innovation program under grant agreement No~802020-ERC-HARMONIC.
This project was partially carried out at the Simons Institute for the Theory of Computing, UC Berkeley, while the authors were participating in the program Analysis and TCS in Summer 2023.

\paragraph{Layout of the paper.} In the remainder of the introduction, we introduce some notation and terminology on systems of linear forms.  In \Cref{section:graphcounting}, we discuss the graph counting lemmas and their proofs. In \Cref{section:KMbounds}, we discuss the Kelley--Meka-type bounds for binary linear systems. In \Cref{section:concluding}, we discuss some concluding remarks regarding the limitations of these methods in relation to other systems of linear forms. 

\subsection{Notation and terminology}\label{subsection:notationandterm}
  The function $\log$ is the base~$2$ logarithm.  For a positive integer $m$, we denote $[m] \defeq \{1,\ldots,m\}$.  We denote random variables with bold font. 

Let $\cX$ be a finite set.
Given a function $f\colon \cX\to \mathbb{R}$ and $p \in [1,\infty)$, let $\|f\|_p = \left( \Ex [|f(\bfx)|^p]\right)^{1/p}$, and let $\|f\|_\infty = \max_x |f(x)|$.  

 We use the standard Bachmann--Landau asymptotic notation: $O(\cdot), \Omega(\cdot), \Theta(\cdot),o(\cdot)$, and $\omega(\cdot)$. Sometimes, we will use these notations in combination with subscripts to indicate that the hidden constants may depend on the parameters in the subscript. For example, $O_m(1)$ means bounded by a constant $c$ that may depend on $m$. We sometimes write $f \lesssim g$ to indicate $f=O(g)$.

\paragraph{Systems of linear forms.} A \emph{linear form} in $d$ variables is a vector $L=(\lambda_1,\ldots,\lambda_d) \in \Z^d$. Given an abelian group $G$, the linear form $L$ defines a linear map $G^d \to G$ by $L(x_1,\ldots,x_d) \defeq \sum_i \lambda_i x_i$. A \emph{system of $m$ linear forms in $d$ variables} is a tuple $\cL=(L_1,\ldots, L_m)$ of linear forms. A tuple $(a_1,\ldots,a_m) \in G^m$ is called an \emph{instance} of $\cL$ if there exists $x \in G^d$ with $\cL(x) \defeq (L_1(x),\ldots,L_m(x))=(a_1,\ldots,a_m)$. It is called a \emph{non-degenerate instance} if additionally $a_1,\ldots,a_m$ are all distinct.

We often require that $|G|$ be coprime with all the coefficients of the linear forms that define $\cL$.  Otherwise, it would be possible to have sets that do not contain instances of even a single linear form. For example, let $L(x) = 2x$ and $G = \mathbb{Z}_m$ for some even number $m$. Then the set $A=\{a\in G: a\bmod 2 \equiv 1 \}$ does not contain any instance of $L(x)$ for $x\in G$.    

For an integer $\lambda$ with $(\lambda,|G|) = 1$, the map  $x \mapsto \lambda x$ is a $G$-automorphism. Let $\lambda^{-1}$ be a positive integer such that $\lambda^{-1}(\lambda x) = x$ for all $x\in G$. This is guaranteed to exist because $(\lambda,|G|)=1$ implies the existence of an integer $\lambda^{-1}$ with $ \lambda^{-1}\lambda \equiv 1\mod |G|$.

Let $\cL=(L_1,\ldots,L_m)$ be a system of linear forms in $d$ variables and let $A\subseteq G$. Let $\bfx_1,\ldots,\bfx_d$ be independent random variables taking values in $G$ uniformly at random. The probability that  $\cL(\bfx_1,\ldots,\bfx_d)\in A^m$ is given by the expectation  
\[t_{\cL}(A) \defeq \Ex \left[ \prod_{i=1}^m A(L_i(\bfx_1,\ldots,\bfx_d)) \right].\]

\paragraph{Translation-invariance.} 
When studying linear patterns,  it is natural to consider systems whose instances are invariant under translations.  

\begin{definition}[Translation-Invariance]\label{definition:ti}
A system  $\cL=(L_1,\ldots, L_m)$ in $d$ variables is \emph{translation-invariant} over a finite abelian group $G$ if for every instance $(a_1,\ldots,a_m)$ of $\cL$ and every $c$, the tuple $(c+a_1,\ldots,c+a_m)$ is also an instance of $\cL$. 
\end{definition}

For example, $(a,a+b,a+2b)$, which represents $3$-progressions, is translation-invariant over every finite abelian group. As another example,  $(x-y,y-z,2x-2z)$ is translation-invariant over $\mathbb{Z}_3$, but it is not translation-invariant over $\mathbb{Z}_5$.

Given a system $\cL=(L_1,\ldots, L_m)$ of linear forms (not necessarily translation-invariant) in $d$ variables $x_1,\ldots,x_d$, we are interested in sets $A \subseteq G$ that are free of \emph{translations} of non-degenerate instances of $\cL$. That is, there are no $c \in G$ and non-degenerate instance $(a_1,\ldots,a_m)$ of $\cL$ such that $(c+a_1,\ldots,c+a_m) \in A^m$. Equivalently, one can introduce an additional variable $x_{d+1}$, and require that $A$ is free of non-degenerate instances of the \emph{translation-invariant} system $\cL'$, where 
\[ \cL'(x_1,\ldots,x_{d+1}) \defeq (x_{d+1}+L_1(x_1,\dots,x_d),\ldots,x_{d+1}+L_m(x_1,\dots,x_d)). \]

\begin{definition}[Binary Systems of Linear forms] 
\label{def:binary}
A  system of linear forms  $\cL=(L_1,\ldots, L_m)$ in the $d$ variables $x_1,\ldots,x_d$  is \emph{binary} if every linear form in $\cL$ is supported on exactly two variables, and moreover no two linear forms in $\cL$ are supported on the same two variables.
\end{definition}

In other words, every  $L_i$ is of the form $\lambda_i x_{a_i} + \eta_i x_{b_i}$ where $a_i<b_i$ and $a_i,b_i \in \{1,\ldots,d\}$, and moreover, $\{a_i,b_i\} \neq \{a_j,b_j\}$ if $i \neq j$.
Such a binary system naturally defines an oriented acyclic graph $H$ with vertex set $\{1,\ldots,d\}$ and an edge $(a_i,b_i)$ for the linear form $L_i =\lambda_i x_{a_i} + \eta_i x_{b_i}$. We say that $H$ is the \emph{underlying graph} of $\cL$.  
In the case of $3$-progressions, the underlying graph $H$ is an oriented triangle, as described below.
\begin{example}[3-progressions as a binary system of linear forms]
Let $N>2$ be an odd number. One can apply the change of variables  $a=2x-2y$ and $b=2y-x-z$ to $(a,a+b,a+2b)$ and capture $3$-progressions in $\Z_N$ with the system of linear forms $\mathcal{L} = (2x-2y,x-z,2y-2z)$. The underlying graph $H$ of $\mathcal{L}$ is the oriented triangle $\{(1,2),(1,3),(2,3)\}$. 
\end{example}

\section{Graph counting lemma}\label{section:graphcounting}

We will adopt a more general framework that encompasses oriented graphs. 
Recall that a directed graph is an \emph{oriented graph} if it contains no opposite pairs of directed edges, or equivalently, it does not contain any $2$-cycles.   Let $H=(V_H,E_H)$ be a fixed oriented graph  with vertex set $V_H=\{1,\ldots,k\}$. Let $\cX$ be a finite set. For every  $A \subseteq \cX \times \cX$, define 
\[t_H(A) \defeq \Ex  \left[\prod_{uv \in E_H} 
A(\bfx_u,\bfx_v)\right],\]
where $\bfx_v$ for $v \in V_H$ are independent random variables taking values in $\cX$ uniformly at random. 
One can interpret $A$ as a directed graph with vertex set $\cX$ and edges $(u,v)\in A$. Then $t_H(A)$ is the probability that a random map from the vertices of $H$ to the vertices of $A$ is a graph morphism, i.e., it maps every edge of $H$ to an edge in $A$. It is worth noting that, more generally, one can define $t_H(f)$ for a function $f\colon \cX\times\cX\to[0,1]$. Our counting lemma is based on the following notion of pseudo-randomness. 

\begin{definition}[Spreadness for graphs]
Let $r\ge 1$ and $ \delta>0$, and let $\cX$ and $\cY$ be finite sets.  A set $A \subseteq \cX \times \cY$ with density $\alpha$ is $(\delta,r)$-\emph{spread} if, for all subsets $S$ and $T$ of density at least $2^{-r}$, we have
\[\Ex_{\substack{\bfx \sim S \\ \bfy \sim T}} [A(\bfx,\bfy)] \le (1+\delta) \alpha.\]
\end{definition}
 
The following counting lemma states that if $A$ is spread and does not have many vertices with small degrees, then $t_H(A) \approx \alpha^m$ for all oriented graphs $H$ on $m$ edges.  We state the result in a more general setting.  

\begin{restatable}[Main graph counting lemma]{thm}{countingLemmaGraphs}
\label{thm:main_graph_count}
Let $H=([k],E)$ be an oriented graph with $m$ edges such that $(u,v) \in E$ implies $u<v$, and let $\alpha>0$.  For every $0<\eps<1$, there exists $\delta=\delta(\eps,m)>0$ such that the following holds. Let $\cX_1,\ldots,\cX_k$ be finite sets, and for every edge $(u,v) \in E$, suppose $A_{uv} \subseteq \cX_u \times \cX_v$ has density $\alpha_{uv} \ge \alpha$. If
\[ 
\left|\Ex \prod_{(u,v) \in E} A_{uv}(\bfx_u,\bfx_v) - \prod_{(u,v)\in E}\alpha_{uv} \right| \ge \eps \prod_{(u,v)\in E}\alpha_{uv},
\] 
then for some  $(u,v) \in E$, at least one of the following cases holds:
\begin{enumerate}
\item \label{itm:main_graph_cound_density_rectangle_increment} There exist $S \subseteq \cX_u$ and $T \subseteq \cX_v$  of densities at least $2^{-O_{m,\eps}(\log^2(2/\alpha))}$ such that   \[\Ex_{\substack{(\bfx,\bfy) \in S\times T}}[A_{uv}(\bfx,\bfy)] \ge (1+\delta) \alpha_{uv} \]
\item \label{itm:main_graph_cound_bounded_degrees_left} There exists $S \subseteq \cX_u$ of density at least $2^{-O_{m,\eps}(\log(2/\alpha))}$ such that  
\begin{equation*}     
    \Ex_{\bfx \in \cX_v}[A_{uv}(z,\bfx)] \le (1-\delta) \alpha_{uv} \qquad \forall z \in S
\end{equation*}
\end{enumerate}
\end{restatable}

In other words, if  $\left|t_H(A)- \alpha^m \right| \ge \eps \alpha^m$, then either $A$ contains a large subgraph that is significantly denser, or it contains many vertices whose degrees are significantly smaller than the average.

Our next theorem shows that when the goal is to identify a rectangle where the density deviates from $\alpha$, we do not need to require any assumption on the degrees of the vertices, and additionally, we can ensure a stronger lower bound on the density of the rectangle.

\begin{restatable}{thm}{countingLemmaDisc}
\label{thm:low_disc_graph_count}
Let $H=([k],E)$ be an oriented graph with $m$ edges such that $(u,v) \in E$ implies $u<v$, and let $\alpha>0$.    Let $\cX_1,\ldots,\cX_k$ be finite sets, and for every edge $(u,v) \in E$, suppose $A_{uv} \subseteq \cX_u \times \cX_v$ has density $\alpha_{uv} \ge \alpha$. 
\begin{itemize}
\item[(i)] If 
\[\Ex \prod_{(u,v) \in E} A_{uv}(\bfx_u,\bfx_v) \ge (1+\eps)^m \prod_{(u,v)\in E}\alpha_{uv},\]
then for some  $(u,v) \in E$, there exist $S\times T \subseteq \cX_u \times \cX_v$ with  $\frac{|S||T|}{|\cX_u||\cX_v|}=\Omega(\eps \alpha^{m+1})$ such that   
\[\Ex_{\substack{(\bfx,\bfy) \sim S\times T}}[A_{uv}(\bfx,\bfy)] \ge \left(1+\frac{\eps}{2}\right) \alpha_{uv}. \]
\item[(ii)] If 
\[\Ex \prod_{(u,v) \in E} A_{uv}(\bfx_u,\bfx_v) \le (1-\eps)^m \prod_{(u,v)\in E}\alpha_{uv},\]
then for some  $(u,v) \in E$, there exist $S\times T \subseteq \cX_u \times \cX_v$ with    $\frac{|S||T|}{|\cX_u||\cX_v|}=\Omega(\eps(1-\eps)^{m-1}\alpha^m)$ such that   
\[\Ex_{\substack{(\bfx,\bfy) \sim S\times T}}[A_{uv}(\bfx,\bfy)] \le \left(1-\frac{\eps}{2}\right) \alpha_{uv}. \]
\end{itemize}
\end{restatable}

This implies the following immediate corollary for bipartite graphs which we use in the proof of \Cref{thm:main_graph_count}.
\begin{corollary}\label{corollary:KLMbipartite}
    Let $H=(L,R,E)$ be a bipartite graph with $m$ edges with left vertices $L=[k]$ and right vertices $R = [p]$, and let $0<\eps<1$. For every bipartite graph $A\colon \cX\times \cY\to \{0,1\}$  and density $\alpha$, the following holds.
Suppose 
 \[\Ex_{\substack{\bfx_1,\ldots,\bfx_k \in \cX \\\bfy_1,\ldots,\bfy_p \in \cY }}  \left[\prod_{\{i,j\} \in E} {
A}(\bfx_i,\bfy_j)\right]\geq (1+\eps)^m\alpha^m.\]
Then there exist $S \subseteq \cX$ and $T\subseteq \cY$ with  $\frac{|S||T|}{|\cX||\cY|}=\Omega(\eps\alpha^{m+1})$ such that   
\[\Ex_{\substack{(\bfx,\bfy) \sim S\times T}}[A(\bfx,\bfy)] \ge \left(1+\frac{\eps}{2}\right) \alpha. \]
\end{corollary}
As mentioned in the introduction,  the case where $H$  is a complete bipartite graph was proven in~\cite{KLM}.

\subsection{Proof of \texorpdfstring{\cref{thm:low_disc_graph_count}}{Theorem \ref{thm:low_disc_graph_count}}}
 
To prove \Cref{thm:low_disc_graph_count}, we need the following lemma, which is a generalization of \cite[Claim 4.6]{KLM}. 
Let $\cX,\cY$ be finite sets, define 
\[ \mathcal{R} \defeq \{S \times T : S \subseteq \cX,T \subseteq \cY\}, \]
and let $\conv(\mathcal{R})$ denote the convex hull of $\mathcal{R}$.
\begin{lemma}
\label{lem:soft_rect}
Let $\Delta \ge 1$  and $\gamma, \eps  \in (0,1]$.  Let $D\colon \cX \times \cY \to \mathbb{R}_{\ge 0}$ and $F \in \conv(\mathcal{R})$, and suppose $\norm{D}_\infty \le \Delta$ and $\norm{F}_1 \ge \gamma$.
\begin{itemize}
\item[(i)]  If $\inp{\frac{F}{\norm{F}_1}}{D} \ge 1+\eps$,
there exists a rectangle $S \times T$ with $\frac{|S||T|}{|\cX||\cY|} \ge \frac{\eps \gamma}{4 \Delta}$ and 
\[\Ex_{(x,y) \in S \times T} D(x,y) \ge 1+ \frac{\eps}{2}.  \]
\item[(ii)] If $\inp{\frac{F}{\norm{F}_1}}{D} \le 1-\eps$,
there exists a rectangle $S \times T$ with $\frac{|S||T|}{|\cX||\cY|} \ge \frac{\eps \gamma}{4 \Delta}$ and 
\[\Ex_{(x,y) \in S \times T} D(x,y) \le 1- \frac{\eps}{2}.  \]
\end{itemize}
\end{lemma}
\begin{proof}
Let $F= \sum_i c_i R_i$ for rectangles $R_i \in \mathcal{R}$ and $c_i \ge 0$ with $\sum_i c_i=1$. Let $\tau=\frac{\eps \gamma}{4\Delta}$, and set $F \defeq \sum_i c'_i R_i$ with 
\[
c_i'
=
\begin{cases}
c_i & \text{if } \norm{R_i}_1 \ge \tau, \\
0 & \text{if } \norm{R_i}_1 < \tau. \\
\end{cases}
\]
Note $\norm{F-F'}_1 \le \tau$.

\medskip

\noindent\textbf{Part (i):} In this case, we have
\[\frac{\inp{F'}{D}}{\norm{F'}_1} \ge \frac{\inp{F'}{D}}{\norm{F}_1} =   \frac{\inp{F}{D}}{\norm{F}_1}+ \frac{\inp{F'-F}{D}}{\norm{F}_1} \ge 
1+\eps - \frac{\norm{F'-F}_1\norm{D}_\infty}{\norm{F}_1} \ge 1+\eps - \frac{\tau \Delta}{\gamma} \ge 1 +\frac{3\eps}{4}.
\] 
The definition of $F'$ implies there is a choice of $R=R_i$ with  $\norm{R}_1 \ge \tau$ and  $\inp{\frac{R}{\norm{R}_1}}{D} \ge 1+\frac{3\eps}{4}$.

\medskip

\noindent\textbf{Part (ii):} We have   
\[\frac{\inp{F'}{D}}{\norm{F}_1} =   \frac{\inp{F}{D}}{\norm{F}_1}+ \frac{\inp{F'-F}{D}}{\norm{F}_1} \le 
1-\eps + \frac{\norm{F'-F}_1\norm{D}_\infty}{\norm{F}_1} \le 1-\eps + \frac{\tau \Delta}{\gamma} \le 1 -\frac{3\eps}{4}.
\] 
Furthermore, 
\[\norm{F'}_1 \ge \norm{F}_1  -\tau \ge \norm{F}_1 - \frac{\eps \norm{F}_1}{4} = \left(1-\frac{\eps}{4}\right) \norm{F}_1.  \]
We get 
\[\frac{\inp{F'}{D}}{\norm{F'}_1} \le \frac{1 -(3\eps/4)}{1-(\eps/4)} \le 1-\frac{\eps}{2}. \]
The definition of $F'$ implies there is a choice of $R=R_i$ with  $\norm{R}_1 \ge \tau$ and  $\inp{\frac{R}{\norm{R}_1}}{D} \le 1-\frac{\eps}{2}$.
\end{proof}

We can now prove \Cref{thm:low_disc_graph_count}.
\begin{proof}[Proof of \Cref{thm:low_disc_graph_count}.]
We will prove the theorem by induction on $m$. The base case $m=1$ is trivial.  

Pick any edge $(u_0,v_0) \in E$ and let $E'  \defeq E - \set{(u_0,v_0)}$ and $H' \defeq (V,E')$.  Let $V' = V \setminus \set{u_0,v_0}$ and let $E''$ be the edges with both endpoints in $V'$. For $(x,y) \in \cX_{u_0} \times \cX_{v_0}$, define 
\[F(x,y) \defeq  \Ex_{\bfx_w \in \cX_w  \ \forall w \in V'} \prod_{(u,v) \in E''} A_{uv}(\bfx_u,\bfx_v) \prod_{(u_0,v) \in E'} A_{u_0v}(x,\bfx_v) \prod_{(u,v_0) \in E'} A_{uv_0}(\bfx_u,y).   \]
Note that $F$ is an average of rectangles: each of the factors may depend on $x$ or $y$ but not both. Therefore, $F \in \conv(\mathcal{R})$. 

\medskip

\noindent\textbf{Part (i):} By applying the induction hypothesis to $H'$, we may assume
\[ \norm{F}_1=\Ex \prod_{(u,v) \in E'} A_{uv}(\bfx_u,\bfx_v) \le (1+\eps)^{m-1} \prod_{(u,v) \in E'} \alpha_{uv},
\]
as otherwise, we are done. Since 
\[\inp{F}{A_{u_0v_0}} = \Ex \prod_{(u,v) \in E} A_{uv}(\bfx_u,\bfx_v) \ge (1+\eps)^m \prod_{uv \in E}\alpha_{uv},\]
by setting $D=\frac{A_{u_0v_0}}{\norm{A_{u_0v_0}}_1}$, we have
\[\inp{\frac{F}{\norm{F}_1}}{D} \ge \frac{(1+\eps)^m \prod_{uv \in E}\alpha_{uv}}{\alpha_{u_0v_0} \norm{F}_1} \ge   1+\eps.\]
Since $\norm{D}_\infty \le \frac{1}{\norm{A_{u_0v_0}}_1} \le \frac{1}{\alpha}$ and by our assumption $\norm{F}_1 \ge  \Ex \prod_{(u,v) \in E} A_{uv}(\bfx_u,\bfx_v) \ge \alpha^m$, \Cref{lem:soft_rect} provides a rectangle $S \times T \subseteq \cX_{u_0} \times \cX_{v_0}$ with $\frac{|S||T|}{|\cX_{u_0}||\cX_{v_0}|} = \Omega(\eps \alpha^{m+1})$ and 
\[\Ex_{(\bfx,\bfy) \in S \times T} A_{u_0v_0}(\bfx,\bfy) \ge (1+ \frac{\eps}{2})\alpha_{u_0v_0}.\]

\medskip

\noindent\textbf{Part (ii):} By applying the induction hypothesis to $H'$, we may assume
\[ \norm{F}_1=\Ex \prod_{(u,v) \in E'} A_{uv}(\bfx_u,\bfx_v) \ge (1-\eps)^{m-1} \prod_{(u,v) \in E'} \alpha_{uv},
\]
as otherwise, we are done. In this case, we have
\[\inp{\frac{F}{\norm{F}_1}}{D}\le \frac{(1-\eps)^m\prod_{(u,v) \in E} \alpha_{uv}}{(1-\eps)^{m-1}\prod_{(u,v) \in E} \alpha_{uv}} \le 1-\eps.\]
Since $\norm{D}_\infty \le \frac{1}{\alpha}$ and $\norm{F}_1  \ge  (1-\eps)^{m-1} \alpha^{m-1}$, \Cref{lem:soft_rect} provides a rectangle $S \times T$ with $\frac{|S||T|}{|\cX_{u_0}||\cX_{v_0}|} = \Omega(\eps(1-\eps)^{m-1} \alpha^{m})$ and 
$\Ex_{(\bfx,\bfy) \in S \times T} A_{u_0v_0}(\bfx,\bfy) \le (1- \frac{\eps}{2})\alpha_{u_0v_0}$.
\end{proof}

\subsection{Proof of \texorpdfstring{\cref{thm:main_graph_count}}{Theorem \ref{thm:main_graph_count}}}

The proof of \cref{thm:main_graph_count} is more involved and requires some preparation.

\paragraph{Grid norms.}  One of the striking results in the work of Chung, Graham and Wilson~\cite{CGW89} states that if an undirected graph with edge density $\alpha$ has approximately the same number of $4$-cycles as the Erd\H{o}s--R\'enyi random graph $G(n,\alpha)$, then it shares several structural, spectral and statistical properties with $G(n,\alpha)$. 

The pseudo-randomness theorem of~\cite{CGW89} naturally generalizes to binary relations $A \subseteq \cX \times \cY$. A set $A \subseteq \cX \times \cY$ exhibits any of  several pseudo-random properties if 
\(t^b_{K_{2,2}}(A)  =\alpha^4 + o(1)\)
where 
\[t^b_{K_{k,p}}(A) \defeq    \Ex_{\substack{\bfx \sim \cX^k \\ \bfy \sim \cY^p}} \left[ \prod_{i=1}^{k} \prod_{j=1}^{p} A(\bfx_i,\bfy_j) \right].\]

One of the key ideas used in the work of Kelley and Meka, and more explicitly in the follow-up work of Kelley, Lovett, and Meka~\cite{KLM}, is to consider a certain generalization of $t^b_{K_{2,2}}(A)$.

\begin{definition}[Grid semi-norms] 
For a function \(f\colon \cX \times \cY \rightarrow \mathbb{R}\) and \(k, p \in \mathbb{N}\), let
\begin{align*}
U_{k,p}(f) &\defeq    \Ex_{\substack{\bfx \sim \cX^k \\ \bfy \sim \cY^p}} \left[ \prod_{i=1}^{k} \prod_{j=1}^{p} f(\bfx_i,\bfy_j) \right].
\end{align*}
Also define \(\|f\|_{U(k,p)} \defeq |U_{k,p}(f)|^{1/p k}\) and \(\|f\|_{|U(k,p)|} \defeq U_{k,p}(|f|)^{1/p k}\).
\end{definition}
\begin{remark}\label{remark:equivbipartite}
Note that if $A:\cX\times \cY\to \{0,1\}$, then  
\(t^b_{K_{k,p}}(A) = U_{k,p}(A)\).
\end{remark}

\cref{lem:decoupling} below is the analogue of the so-called Gowers--Cauchy--Schwarz inequality for $U(k,p)$ norms. 

\begin{lemma}\label{lem:decoupling}(Gowers--Cauchy--Schwarz inequality)
Let $p,k \in \N$, and for $i \in [k]$ and $j \in [p]$, let $f_{ij}\colon \cX_i \times \cY_j \to \R$. Suppose that either $k,p$ are both even numbers, or the functions $f_{ij}$ take only non-negative values.
Then  
\begin{equation}
\label{eq:Gowers-CS}
\Ex_{\bfx_1\sim \cX_1,\ldots,\bfx_k\sim \cX_k} \Ex_{\bfy_1\sim \cY_1,\ldots,\bfy_p\sim\cY_p} \left[ \prod_{i=1}^k 
\prod_{j=1}^p f_{ij}(\bfx_i,\bfy_j) \right] \le \prod_{i=1}^k \prod_{j=1}^p \norm{f_{ij}}_{U(k,p)}. 
\end{equation}

\end{lemma}
\begin{proof}
Recall the generalized H\"older inequality:
\[
 \Ex_\bfz \left| \prod_{j=1}^p g_j(\bfz) \right| \leq
 \prod_{j=1}^p \left(\Ex_\bfz |g_j(\bfz)|^p\right)^{1/p}.
\]
Choosing $z = (x_1,\dots,x_k)$ and $g_j(x_1,\dots,x_k) = \Ex_{\bfy_j\sim \cY_j} \prod_i f_{ij}(x_i,\bfy_j)$, we obtain  
\begin{equation}
\label{eq:Gowers-CS-proof}
\text{L.H.S.\ of \eqref{eq:Gowers-CS}} \le 
\prod_{j=1}^p \left(\Ex_{\bfx_1\sim \cX_1,\ldots,\bfx_k\sim \cX_k} \left(\Ex_{\bfy_j \sim \cY_j}\prod_{i=1}^k  f_{ij}(\bfx_i,\bfy_j) \right)^p\right)^{1/p}.
\end{equation}
Next, note that for every fixed $j$, we can apply the generalized H\"older inequality to the variable $\bfy \sim \cY_j^p$ and obtain 
\begin{align*}
\Ex_{\bfx_1\sim \cX_1,\ldots,\bfx_k\sim \cX_k} \left(\Ex_{\bfy \sim \cY_j}\prod_{i=1}^k  f_{ij}(\bfx_i,\bfy) \right)^p &= 
\Ex_{\bfy \sim \cY_j^p} \Ex_{\bfx_1\sim \cX_1,\ldots,\bfx_k\sim \cX_k} 
 \prod_{i=1}^k \prod_{t=1}^p f_{ij}(\bfx_i,\bfy_t) \\ 
 &= 
\Ex_{\bfy \sim \cY_j^p} 
\prod_{i=1}^k
\left(\Ex_{\bfx_i\sim \cX_i} 
  \prod_{t=1}^p f_{ij}(\bfx_i,\bfy_t)
  \right)
  \\ 
&\le \prod_{i=1}^k \left(\Ex_{\bfy \sim \cY_j^p} \left(\Ex_{\bfx \sim \cX_i}  \prod_{t=1}^p f_{ij}(\bfx,\bfy_t) \right)^k\right)^{1/k} \\
&= \prod_{i=1}^k \norm{f_{ij}}_{U(k,p)}^p.
\end{align*}
Substituting these in \eqref{eq:Gowers-CS-proof}, we obtain the desired inequality.  \end{proof}

Applying \cref{lem:decoupling} to the expansion of $U_{k,p}(f+g)$ easily implies that $\norm{\cdot}_{|U(k,p)|}$ is a norm, and that when $p$ and $k$ are both even, $\norm{\cdot}_{U(k,p)}$ is a semi-norm.

\begin{corollary}[\protect{\cite[Claim 4.2]{KLM}}]\label{cor:monotonicity_of_grid_semi_norm}
    If $f\colon \cX\times\cY \to \R_{\geq 0}$ then $\norm{f}_{U(k,p)}$ is monotone in both $k$ and $p$.
\end{corollary} 
\begin{proof}
   We will prove monotonicity for $k$; the same argument can be applied to $p$. 
   Let $k'  \ge  k$, and define $f_{ij}\colon \cX \times \cY \to \R_{\geq 0}$ for $1 \le i \le k'$ and $1 \le j \le p$ as 
   \[f_{ij}
   =\begin{cases}
   f &   i \le  k,  \\
   1 & i > k.
   \end{cases}
   \]
   Applying \eqref{eq:Gowers-CS} to these functions, we have   
    \[\norm{f}_{U(k,p)}^{kp}= \Ex_{\bfx_1,\ldots,\bfx_{k'}} \Ex_{\bfy_1,\ldots,\bfy_p} \left[ \prod_{i=1}^{k'} 
\prod_{j=1}^p f(\bfx_i,\bfy_j) \right] \le  \norm{f}_{U(k',p)}^{kp} \norm{1}_{U(k',p)}^{k'p-kp}= \norm{f}_{U(k',p)}^{kp}, \]
as desired. 
\end{proof}

We also need the following simple lemma.

\begin{lemma}\label{lem:largeCenteredNormToHP}
Let $p \ge 1$, and let $f\colon \cX \to [0,1]$ satisfy $\Ex[f] = \alpha$ and $\norm{f-\alpha}_{p} \ge \eps \alpha$. Then one of the following two cases holds:
\begin{itemize}
    \item [(i)]   $\Ex_\bfx \left[ \1_{\left[f(\bfx)>\alpha\left(1+\frac{\eps}{4}\right)\right]}\right] \geq \eps^p\alpha^p/4 $.
    \item[(ii)] $\Ex_\bfx \left[ \1_{\left[f(\bfx)<\alpha\left(1-\frac{\eps}{4}\right)\right]}\right]\geq \eps^p/4$.
\end{itemize}

\end{lemma}
\begin{proof}
From the assumption, 
\[\Ex_\bfx |f(\bfx)-\alpha|^p   = \norm{f-\alpha}_{p}^p \ge \eps^p \alpha^p.\]
We   have 
\[ \Ex \left[|f(\bfx)-\alpha|^p \1_{\left[|f(\bfx)-\alpha| \le \frac{\eps\alpha}{4}\right]} \right] \le  (\eps/4)^p \alpha^p \le \frac{\eps^p}{4} \alpha^p, \]
and therefore, 
\begin{align*}
\frac{3\eps^p}{4} \alpha^p &\le  \Ex \left[|f(\bfx)-\alpha|^p \1_{\left[f(\bfx)>\alpha\left(1+\frac{\eps}{4}\right)\right]}\right]+ \Ex \left[|f(\bfx)-\alpha|^p \1_{\left[f(\bfx)<\alpha\left(1-\frac{\eps}{4}\right)\right]}\right] \\
&\le \Ex \left[ \1_{\left[f(\bfx)>\alpha\left(1+\frac{\eps}{4}\right)\right]}\right]+ \alpha^p\Ex \left[ \1_{\left[f(\bfx)<\alpha\left(1-\frac{\eps}{4}\right)\right]}\right],
\end{align*}
which implies that one of the two cases (i) or (ii) must hold.
\end{proof}

\paragraph{Main technical lemma.}
Next, we prove the following lemma, which constitutes the technical core of the proof of \cref{thm:main_graph_count}.

\begin{lemma}[Main technical lemma]
\label{lem:MainTechnical}
Let $H=([k],E)$ be an oriented graph with $m$ edges such that $(u,v) \in E$ implies $u<v$. For every $0<\eps<1$, there exists $\delta=\delta(\eps,H) \ge \eps^2 2^{-4 m-10}$ such that the following holds. Let $d$ be the maximum indegree of a vertex in $H$ and let $\alpha>0$.  Let $\cX_1,\ldots,\cX_k$ be finite sets, and for every edge $(u,v) \in E$, let $f_{uv}\colon \cX_u \times \cX_v \to \{0,1\}$ with $\Ex[f_{uv}]=\alpha_{uv} \ge \alpha$ be such that 
\begin{equation}
\label{eq:LargeNorm_assumption}
\left|\Ex \prod_{(u,v) \in E} f_{uv}(\bfx_u,\bfx_v) - \prod_{(u,v)\in E}\alpha_{uv} \right| \ge \eps \prod_{(u,v)\in E}\alpha_{uv}.
\end{equation}
There exists $(u,v) \in E$ and $p \lesssim \frac{m}{\eps} \log(2/\alpha)$ such that at least one of the following holds:
\begin{itemize}
\item[(i)] 
\begin{equation}\label{eq:LargeNormForDensityIncMulti}
    \norm{f_{uv}}_{U(2(d-1),p)} \geq (1+\delta) \alpha_{uv}.
\end{equation} 
\item[(ii)] 
\begin{equation}\label{equation:multipartitecase2}
\norm{\Ex_{\bfx_v \sim \cX_v}f_{uv}( \cdot,\bfx_v) - \alpha_{uv}}_{p} \geq \delta \alpha_{uv}. 
\end{equation}
\end{itemize}
\end{lemma}
\begin{proof}
We can assume that $H$ contains no isolated vertices.
We can also assume that $H$ contains no isolated edges. Indeed, suppose that $H$ contains an isolated edge $e$, and let $H'$ be obtained from $H$ by removing $e$. If \eqref{eq:LargeNorm_assumption} holds for $H$ then it also holds for $H'$. In this way, we can remove all isolated edges from $H$.

We will prove the theorem by induction on the quantity $\kappa(H) \defeq 2|E|-d_1(H)$, where $d_1(H)$ is the number of vertices of degree $1$ in $H$. More precisely, the induction hypothesis states that the theorem holds with $\delta(\eps,H) = \eps^2 2^{-2\kappa(H)-10}$.

To verify the base case $\kappa(H)=0$, note that in this case,  $H$ is empty, and so the L.H.S.\ of \eqref{eq:LargeNorm_assumption} is always $0$, and the statement is trivially true.
 
We normalize the functions $f_{uv}\colon \cX_u \times \cX_v \to \{0,1\}$ by defining $F_{uv}\colon \cX_u \times \cX_v \to\{0,\alpha_{uv}^{-1}\}$ as $F_{uv} \defeq \alpha_{uv}^{-1}\cdot f_{uv} $. We can rephrase the assumption, \eqref{eq:LargeNorm_assumption}, as 
\begin{equation}
\label{eq:LargeNorm_assumption_normalized}
\left|\Ex \prod_{uv \in E} F_{uv}(\bfx_u,\bfx_v) - 1 \right| \ge \eps. 
\end{equation}

Consider the vertex $ s=k$, which is a sink in the graph since $H$ is topologically ordered.
Let $u_1,\ldots,u_\ell$ be the neighbors of $s$, where $1\le \ell \le d$.
Let $H'=(V\setminus\{s\},E')$ where $E' \subsetneq E$ is the set of the edges that do not involve $s$. 

Since $H$ has no isolated edges,   $d_1(H')\geq d_1(H)-\ell$, and so \[\kappa(H')=2|E'|-d_1(H')\le 2(|E|-\ell)-(d_1(H)-\ell)=\kappa(H)-\ell.\]
By applying the induction hypothesis to $H'$, we may assume 
\begin{equation}
    \label{eq:IH_remove_one_node}
    \left| \Ex \prod_{uv \in E'} F_{uv}(\bfx_u,\bfx_v) -1 \right| \le \frac \eps 2,
\end{equation}
as otherwise, we are done with 
\[\delta=\delta(\eps/2,H')= (\eps/2)^2 2^{-2{\kappa(H')}-10} \ge \eps^2 2^{-2{\kappa(H)}-10}=\delta(\eps,H),
\] as desired.
In particular, we may assume 
\[ \Ex \prod_{uv \in E'} F_{uv}(\bfx_u,\bfx_v)  \le 2. \]
Therefore, using $F_{uv} \in \{0, \alpha_{uv}^{-1}\}$, we have for every $q \ge 1$,
\begin{equation}
    \label{eq:subgraph_moment_bound}
    \Ex\left[ \left(\prod_{uv \in E'} F_{uv}(\bfx_u,\bfx_v)\right)^q \right]=
 \Ex \left[ \prod_{uv\in E'}\alpha_{uv}^{-(q-1)} \prod_{uv \in E'} F_{uv}(\bfx_u,\bfx_v) \right] \leq
 2\cdot \prod_{uv\in E'}\alpha_{uv}^{-(q-1)}\leq 2 \cdot \alpha^{m(1-q)}.
\end{equation}
We will single out the edge $u_1s$ and set $E''\defeq E \setminus \{u_1s\}$. 
\paragraph{Case 1: $\ell=1$.}
In this case, $E''=E'$. 
 Combining \eqref{eq:IH_remove_one_node} and \eqref{eq:LargeNorm_assumption_normalized}, we have
\[ \left|\Ex \left[(F_{u_1s}(\bfx_{u_1},\bfx_s)-1)\prod_{uv \in E'} F_{uv}(\bfx_u,\bfx_v)\right] \right|= \left|\Ex \prod_{uv \in E} F_{uv}(\bfx_u,\bfx_v) -   \Ex \prod_{uv \in E'} F_{uv}(\bfx_u,\bfx_v)\right|\ge \frac \eps 2. \]
Consequently, we may apply H\"older's inequality over the variables $\bfx \defeq \{\bfx_u\}_{u \in V \setminus \{s\}}$ with the even integer $p=2m\ceil{\log (2\alpha^{-1})}\ge 2m\log (2\alpha^{-1})$ and $\frac{1}{p}+\frac{1}{q}=1$, obtaining 
\begin{align*}
     \frac \eps 2\le &\left|\Ex \left[(F_{u_1s}(\bfx_{u_1},\bfx_s)-1)\prod_{uv \in E'} F_{uv}(\bfx_u,\bfx_v)\right]\right|\\
    &  \le \Ex_{\bfx} \left[\left(  \prod_{uv \in E'} F_{uv}(\bfx_u,\bfx_v)\right)^q \right]^{1/q}  \times   \Ex_\bfx \left[\left( \Ex_{\bfx_s} (F_{u_1s}(\bfx_{u_1},\bfx_s)-1) \right)^{p} \right]^{1/p} \\
    &\le 3\cdot\Ex_\bfx \left[\left( \Ex_{\bfx_s} (F_{u_1s}(\bfx_{u_1},\bfx_s)-1)  \right)^{p} \right]^{1/p},
\end{align*}
where the last inequality follows from
\begin{equation}
\label{eq:Holder_alpha}
 \left(2\cdot\alpha^{m(1-q)}\right)^{1/q} \leq
 2\cdot\alpha^{m\left(\frac{1}{q}-1\right)} =
 2\cdot\alpha^{-\frac{m}{p}} \leq 2\cdot\alpha^{-\frac{m}{2m\log(2\alpha^{-1})}} \leq 3.
\end{equation}
We conclude that 
\[\|\Ex_{\bfy} F_{u_1s}(\cdot,\bfy)-1\|_p\ge  \eps/ 6 \ge \delta,\]
which verifies~\cref{equation:multipartitecase2}. 
\paragraph{Case 2: $\ell \geq 2$.}
Consider the graph $H''$ obtained from $H$ by ``detaching'' the edge $u_1s$ from the node $s$. More formally, to obtain $H''$, we add a new vertex $s'$ to $H$, remove the edge $u_1s$, and add the edge $u_1s'$ instead.  Note that $\kappa(H'') \le \kappa(H)-1$.  By applying the induction hypothesis to $H''$ with parameter $\eps/2$, we may assume that \begin{equation}
\label{eq:IH-indegrees}
\left|\Ex_{\bfx}\left[ \left( \prod_{uv \in E''} F_{uv}(\bfx_u,\bfx_v)\right)\cdot\left( \Ex_{\bfz \sim \cX_s} F_{u_1s}(\bfx_{u_1},\bfz) \right)\right]  - 1 \right| \le \frac{\eps}{2}, 
\end{equation}
as otherwise, the assertion of the theorem follows from the induction hypothesis with parameter  $\delta=\delta(\eps/2,H'')\geq \delta(\eps,H)$,
as desired.
Set $J(x,y) \defeq F_{u_1s}(x,y)- \Ex_{\bfz \sim \cX_s }  F_{u_1s}(x,\bfz)$. We have
\begin{multline*}
     \Ex \left[J(\bfx_{u_1},\bfx_s)\prod_{uv \in E''} F_{uv}(\bfx_u,\bfx_v)\right] =\\ \Ex \prod_{uv \in E} F_{uv}(\bfx_u,\bfx_v) -   \Ex_{\bfx }\left[\left(  \prod_{uv \in E''} F_{uv}(\bfx_u,\bfx_v)\right) \cdot\Ex_{\bfz \sim \cX_s} F_{u_1s}(\bfx_{u_1},\bfz) \right].
\end{multline*}
Combining the above equation with \eqref{eq:LargeNorm_assumption_normalized} and \eqref{eq:IH-indegrees}, we obtain  
\[ \left|\Ex \left[J(\bfx_{u_1},\bfx_s)\prod_{uv \in E''} F_{uv}(\bfx_u,\bfx_v)\right] \right| \ge \frac{\eps}{2}. \]
We may apply H\"older's inequality over the variables $\bfx \defeq \{\bfx_u\}_{u \in V \setminus \{s\}}$ with the even integer $p=2m\ceil{\log (2\alpha^{-1})}\ge 2m\log (2\alpha^{-1})$ and $\frac{1}{p}+\frac{1}{q}=1$, obtaining 
\begin{align*}
\Bigg|\Ex \Bigg[J(\bfx_{u_1},\bfx_s) &\prod_{uv \in E''} F_{uv}(\bfx_u,\bfx_v)\Bigg] \Bigg| 
 \\ 
&=\left|\Ex_{\bfx}  \left[\left( \prod_{uv \in E'} F_{uv}(\bfx_u,\bfx_v) \right)\cdot \Ex_{\bfx_s} \left[J(\bfx_{u_1},\bfx_s)\cdot  \prod_{uv \in E'' \setminus E'} F_{uv}(\bfx_u,\bfx_v) \right] \right]\right|  
\\ &\le \Ex_{\bfx} \left[\left(  \prod_{uv \in E'} F_{uv}(\bfx_u,\bfx_v)\right)^q \right]^{1/q}  \times 
   \Ex_\bfx \left[\left( \Ex_{\bfx_s} J(\bfx_{u_1},\bfx_s)\cdot\prod_{i=2}^\ell F_{u_is}(\bfx_{u_i},\bfx_s)  \right)^{p} \right]^{1/p} \\
 &\stackrel{(\ast)}\le 3 \cdot\Ex_\bfx \left[\left( \Ex_{\bfx_s} J(\bfx_{u_1},\bfx_s)\cdot\prod_{i=2}^\ell F_{u_is}(\bfx_{u_i},\bfx_s)  \right)^{p} \right]^{1/p},
\end{align*}
where $(\ast)$ follows from \eqref{eq:subgraph_moment_bound} and \eqref{eq:Holder_alpha}.
We conclude that
\begin{equation}
\label{eq:J_and_F_large}
 \Ex_\bfx \left[\left( \Ex_{\bfx_s} J(\bfx_{u_1},\bfx_s)\prod_{i=2}^\ell F_{u_is}(\bfx_{u_i},\bfx_s)  \right)^{p} \right]^{1/p}
 \ge \frac{\eps}{6}.
\end{equation}
On the other hand, by Cauchy--Schwarz and then \eqref{eq:Gowers-CS}, we have
\begin{align*}
\Ex_\bfx \Bigg[\Bigg( \Ex_{\bfx_s} J(\bfx_{u_1},\bfx_s)\prod_{i=2}^\ell F_{u_is}&(\bfx_{u_i},\bfx_s)  \Bigg)^{p}\Bigg]\\
&= \Ex_{\bfy_1,\ldots,\bfy_p} \Ex_{\bfx_{u_1}} \prod_{j=1}^p J(\bfx_{u_1},\bfy_j) \Ex_{\bfx_{u_2},\ldots,\bfx_{u_\ell}} \prod_{i=2}^\ell\prod_{j=1}^p F_{u_is}(\bfx_{u_i},\bfy_j) \\ 
&\le \norm{J}^p_{U(2,p)} \cdot\left[\Ex_{\bfy_1,\ldots,\bfy_p} \left(\Ex_{\bfx_2,\ldots,\bfx_\ell} \prod_{i=2}^\ell\prod_{j=1}^p F_{u_is}(\bfx_i,\bfy_j)\right)^2\right]^{1/2} \\ 
&\le  \norm{J}^p_{U(2,p)}   \prod_{i=2}^\ell \norm{F_{u_is}}_{U(2(\ell-1),p)}^{p}.
\end{align*}
Taking the $p$-th root of this upper bound and combining it with \eqref{eq:J_and_F_large}, we get
\[ \norm{J}_{U(2,p)}   \prod_{i=2}^\ell \norm{F_{u_is}}_{U(2(\ell-1),p)} \ge \frac{\eps}{6}. \]
If there exists $i \in \{2,\ldots,\ell\}$ with $\norm{F_{u_is}}_{U(2(\ell-1),p)} \ge 2$, we are done by monotonicity of grid norms (\cref{cor:monotonicity_of_grid_semi_norm}) with $2(d-1) \ge 2(\ell-1)$, and otherwise, 
\[ \norm{J}_{U(2,p)}  \ge \frac{\eps}{6 \cdot 2^{\ell-1}}. \]
Note that $\Ex_{\bfz\sim \cX_s} J(x,\bfz) =0$ for every $x \in \cX_{u_1}$. Therefore, we can apply \cref{lem:PSD} (stated and proved below) to $J$ and obtain that for $p' =2p\ceil{1/\eps}$, we have
 \[\norm{1+J}_{U(2,p')}  \ge 1+\frac{\eps^2}{45 \cdot 2^{2\ell}}\ge 1+2\delta.\]
The last inequality is true because  $\kappa(H)=\sum_{v\in V\colon\deg(v)>1}\deg(v)\ge \deg(s)=\ell$.
Since $J(x,y) = F_{u_1s}(x,y)- D(x,y)$ with $D(x,y) \defeq \Ex_{\bfz \sim \cX_s }  F_{u_1s}(x,\bfz)$, applying the triangle inequality,  we have 
\[\norm{F_{u_1s}}_{U(2,p')}+ \norm{D - 1}_{U(2,p')}  \ge 1+2\delta.\]
Now either \[\norm{F_{u_1s}}_{U(2,p')} \ge 1+\delta,  \]
in which case, we are done, or 
\[\norm{D - 1}_{U(2,p')} \ge \delta. \]
Since $D(x,y)-1$ does not depend on $y$, we have $\norm{D - 1}_{U(2,p')} = \norm{D - 1}_2$, and thus in this latter case, we get
\[ \norm{D - 1}_2 \ge \delta, \]
which verifies~\cref{equation:multipartitecase2}.      
\end{proof}

It remains to prove \Cref{lem:PSD} below.

\begin{lemma}
    \label{lem:PSD} If $g\colon \cX \times \cY \to \R$ satisfies $\norm{g}_{U(2,p)}^2  \ge \eps$ for some even $p$ and some $\eps\in(0,1/2]$   and 
    \[\Ex_\bfz g(x,\bfz) =0 \qquad  \forall x \in \cX,\]
    then for every even $p' \ge  \frac{2p}{\eps}$, we have
    $\norm{1+g}_{U(2,p')}  \ge 1+\eps/5$.
\end{lemma}

The proof will make use of the following result of Kelley and Meka.
\begin{proposition}[\protect{\cite[Proposition~5.7]{KM23}}]\label{prop:odd_moments}
 Suppose $f\colon \Omega\to\R$ is such that
    \begin{itemize}
        \item $\Ex (f-1)^k \geq 0$ for all odd $k\in\N$, and
        \item $\norm{f-1}_{k_0}\geq \eps$ for some even $k_0\geq 2$ and some $\eps\in[0,1/2]$.
    \end{itemize}
    Then, for any integer $k'\geq 2k_0/\eps$,
    \[\norm{f}_{k'}\geq 1+\frac{\eps}{2}.\]
\end{proposition}

\begin{proof}[Proof of \cref{lem:PSD}]
Let $g\circ g\colon \cX \times \cX \to \R$ be defined as 
$(g \circ g)(x_1,x_2) \defeq \Ex_\bfy g(x_1,\bfy)g(x_2,\bfy)$. Note that 
\[ \norm{g}_{U(2,p)}^2  = \norm{g \circ g}_p. \]
Moreover, for every $q \in \N$, 
\[ \Ex (g \circ g)^q = \Ex_{\bfy_1,\ldots,\bfy_q} \left(\Ex_\bfx \prod_{i=1}^q g(\bfx,\bfy_i) \right)^2 \ge 0. \]
It follows from  \cref{prop:odd_moments} that for every integer $p' \ge \frac{2p}{\eps}$, 
\[ \norm{g\circ g+1}_{p'} \ge 1+ \frac{\eps}{2}. \]
Moreover, $\Ex_\bfy g(\cdot,\bfy) \equiv 0$ implies that $(g+1)\circ (g+1)=(g \circ g)+1$, and thus 
\[ \norm{1+g}_{U(2,p')} = \norm{(g+1) \circ (g+1)}_{p'}^{1/2} \ge \sqrt{1+ \frac{\eps}{2}}\geq 1+\frac{\eps}{5}. \qedhere \]
\end{proof}

Finally, we are ready to deduce  \cref{thm:main_graph_count} from  \cref{lem:MainTechnical} (main technical lemma). We repeat the statement of  \cref{thm:main_graph_count} for convenience.

\countingLemmaGraphs*

\begin{proof}

By applying \cref{lem:MainTechnical} to $f_{uv} \defeq A_{uv}$ and using the monotonicity of the grid norms, for some   $p\leq \frac{8m}{\eps} \log \frac{2}{\alpha}$, $\delta=\eps^22^{-4m-10}$ and $(u,v)\in E$, at least
one of  the following holds:
\begin{itemize}
\item[(i)]
\[
\|A_{uv} \|_{U(2k,p)} \ge (1+\delta)\alpha_{uv};
\]
\item [(ii)]
\[
 \norm{\Ex_\bfy\left[ A_{uv}(\cdot,\bfy)\right] - \alpha_{uv}}_p  \ge \delta\alpha_{uv}.
\]
\end{itemize}

\begin{itemize}
    \item[\textbf{Case (i)}] 
    Note that by \Cref{remark:equivbipartite}, \[\|A_{uv} \|_{U(2k,p)} = t^b_{K_{2k,p}}(A_{uv}).\]
    In this case, we can apply \Cref{corollary:KLMbipartite} with $H = K_{2k,p}$, which implies the existence of $S \times T \subseteq \cX_u \times \cX_v$ 
with 
\[\frac{|S||T|}{|\cX_u||\cX_v|} \geq  \Omega(\delta\alpha^{2pk+1})
\ge 2^{-O\left(\frac{m^2}{\eps}\log^2(2/\alpha)\right)},\]
that satisfies \cref{itm:main_graph_cound_density_rectangle_increment} of \cref{thm:main_graph_count}.
  \item[\textbf{Case (ii)}]  By applying \cref{lem:largeCenteredNormToHP} to the function $f\colon\cX_u\to[0,1]$ given by 
$f(x)\coloneqq\Ex_\bfy[A_{uv}(x,\bfy)]$, we conclude that one of the following two cases holds: 
\begin{itemize}
    \item[a)] $\Ex \left[ \1_{\left[f(x)>\alpha_{uv}\left(1+\frac{\delta}{4}\right)\right]}\right]\geq (\delta\alpha_{uv})^p/4 = 2^{-O\left(\frac{m^2}{\eps^2}\log^2(2/\alpha)\right)}$.
    \item[b)] $\Ex \left[ \1_{\left[f(x)<\alpha_{uv}\left(1-\frac{\delta}{4}\right)\right]}\right] =\delta^p/4\ge \alpha^{O(m^2/\eps^2)}$. 
\end{itemize}

In case (a), let $S \defeq \{x\in\cX_u : \Ex_\bfy[A_{uv}(x,\bfy)]>\alpha_{uv}(1+\delta/4)\}$ and $T = \cX_v$ and conclude \cref{itm:main_graph_cound_density_rectangle_increment} of \cref{thm:main_graph_count}: \[\Ex_{(\bfx,\bfy)\in S\times T} A_{uv}(\bfx,\bfy)\geq (1+\delta/4)\alpha_{uv}.\]

In case (b), let $S=\{x\in\cX_u : \Ex_\bfy[A_{uv}(x,\bfy)] <\alpha_{uv}(1-\delta/4)\}$ and conclude \cref{itm:main_graph_cound_bounded_degrees_left} of \cref{thm:main_graph_count}:  \[\Ex_{\bfx \in \cX_v}[A_{uv}(z,\bfx)] \le (1-\delta/4) \alpha_{uv} \qquad \text{for all } z \in S. \qedhere \]

\end{itemize}
\end{proof}

\section{Kelley--Meka-type bounds for binary systems}\label{section:KMbounds}

In this section, we establish   Kelley--Meka-type bounds for the density of sets that do not contain translations of non-degenerate instances of binary systems of linear forms.

\GeneralizedKelleyMekaAbelian* 
 
The rest of this section is devoted to the proof of \cref{thm:L-freeAbelian}.

\subsection{Preliminaries}\label{sec:notation}

This section discusses the necessary notation and terminology from additive number theory.

\paragraph{Set addition.}
Let $G$ be a finite abelian group. For an integer $k \ge 0$ and $x\in G$, let 
\[k x = \underbrace{x + x + \dots + x}_{k \text{ times}} \ \text{ and } \ (-k)x=-kx.\]
We denote the sumset of two sets $A,B \subseteq G$ as 
\[A + B \defeq \set{a+b : a\in A, b \in B}. \]
Given a natural number $k \in \mathbb{N}$, we denote the iterated sumset of $A$ as 
\[kA \defeq \set{a_1+\dots+a_k : a_1,\dots,a_k \in A}, \]
and  we let 
\[k\cdot A \defeq \set{ka : a \in A}. \]

 \paragraph{Bohr sets.}
We start with some standard notation and some basic facts about Bohr sets. Let $G$ be a finite abelian group and let $\widehat{G}$  denote the character group of $G$.

Given a set $B \subseteq G$, we define $\mu_B$ to be the uniform probability measure on $B$. 
The relative density of $A \subseteq G$ in $B$ is denoted by $\mu_B(A) = \frac{|A \cap B|}{|B|}$. The normalized indicator function of $B$ is defined as $\varphi_B(\cdot) = \frac{|G|}{|B|}B(\cdot)$ so that $\|\varphi_B\|_1 = 1$.  

Given two functions $f,g\colon G\to\mathbb{R}$, define the convolution and the cross-correlation as 
\[f*g(x) = \Ex_\bfy f(\bfy)g(x-\bfy) \text{ and }f\star g(x) = \Ex_\bfy f(\bfy)g(x+\bfy),\] 
respectively. Note that  
\[\langle f*g, h\rangle = \langle f , g \star h\rangle.\]
We recall the definition of a Bohr set. 
\begin{definition}[Bohr sets] 
For a non-empty $\Gamma \subseteq \widehat{G}$ and $\tau \in [0,2]$,  define the corresponding \emph{Bohr set} $B= \Bohr(\Gamma,\tau)$ as
\[ \Bohr(\Gamma,\tau) \defeq \set{x \in G : |\chi(x)-1|\le \tau \ \ \forall \chi \in \Gamma }. \] 
We call $\Gamma$ the \emph{frequency set} of $B$, and $\tau$ the \emph{width}. We define the \emph{rank} of $B$, denoted by $\rank(B)$, to be the size of $\Gamma$. 
\end{definition}

When we speak of a Bohr set, we implicitly refer to the triple $(\Gamma, \tau, \Bohr(\Gamma,\tau))$, since the set $\Bohr(\Gamma,\tau)$ alone does not uniquely determine the frequency set nor the width.

We note some easy facts about Bohr sets. We have  
\begin{equation}
\label{claim:bohrIntersection}
\Bohr(\Gamma,\tau) \cap \Bohr(\Gamma',\tau)=\Bohr(\Gamma \cup \Gamma',\tau),
\end{equation}
and 
\[\Bohr(\Gamma,\tau)+\Bohr(\Gamma,\tau') \subseteq \Bohr(\Gamma,\tau+\tau').\]
In particular, for any $\lambda \in \mathbb{N}$, we have the following inclusion: 
\[\lambda \cdot \Bohr(\Gamma,\tau) \subseteq \lambda \Bohr(\Gamma,\tau)
\subseteq \Bohr(\Gamma,\lambda\tau).\]
Since $\chi(\lambda x)=\chi(x)^\lambda$, we also have 
\begin{equation}
\label{eq:Bohr_multiply}
\lambda \cdot \Bohr(\Gamma,\tau) = \Bohr(\lambda \cdot \Gamma,\tau).
\end{equation}

\begin{definition}[Dilate of a Bohr set]
If $B=\Bohr(\Gamma,\tau)$ is a Bohr set, and $\rho\ge 0$, then $B_\rho=\Bohr(\Gamma,\rho\tau)$ is the \emph{dilate} of $B$ with parameter $\rho$.
\end{definition}

While Bohr sets, in general, lack approximate group-like properties, Bourgain~\cite{MR1726234} observed that certain Bohr sets are
approximately closed under addition in a certain weak sense. 

\begin{definition}[Regularity] A Bohr set $B$ of rank $d$ is \emph{regular} if for all $\kappa$ with  $|\kappa| \le \frac{1}{100d}$, we have
\[(1 - 100d |\kappa|)|B| \le |B_{1+\kappa}| \le (1 + 100d |\kappa|)|B|.\]
\end{definition}

Note that by \eqref{eq:Bohr_multiply}, if $(\lambda, |G|) = 1$ and $\Bohr(\Gamma,\tau)$ is a regular Bohr set of rank $d$, then $\lambda \cdot \Bohr(\Gamma,\tau)$ is also a regular Bohr set of rank $d$ and has the same size as $\Bohr(\Gamma,\tau)$.  We will use this fact frequently in the proofs.

The following two lemmas show that a low-rank Bohr set $B$ can be made into a regular Bohr set $B_\rho$ without decreasing its size by much. 

\begin{lemma}[{\cite{MR1726234}; See \cite[Lemma 4.25]{TaoVu}}]
\label{lem:BohrRegular}
    For any Bohr set $B$, there exists $\rho \in \left[\frac{1}{2}, 1\right]$ such that $B_\rho$ is regular.
\end{lemma}
\begin{lemma}[{\cite[Lemma 4.20]{TaoVu}}]
\label{lem:Bohr_Size}
    If $\rho \in (0,1)$ and $B$ is a Bohr set of rank $d$ then $|B_\rho| \geq \left(\frac{\rho}{4}\right)^d |B|$.
\end{lemma}

One of the nuances of working with Bohr sets is that even if $B$ is a regular Bohr set (e.g., a large interval centred at $0$), then for randomly chosen $\bfx,\bfx' \in B$, the distribution of $\bfx+\bfx'$ is very different from the uniform distribution on $B$. A common strategy to address this issue is to work with two Bohr sets. We will have an additional regular Bohr set $B'=B_{\rho}$ for some sufficiently small $\rho$ (e.g., a smaller interval centred at $0$). Then, for every $y_0 \in B'$,  when $\bfx \in B$ is chosen randomly and uniformly, the distribution of $\bfx+y_0$ is close to the uniform distribution on $B$.  

 \begin{lemma}[{\cite[Lemma 4.5]{MR3953879}}]
\label{lem:BohrSetSumUniform}
If $B$ is a regular Bohr set of rank $d$ and $ x \in B_\rho$ with $\rho \in (0,1)$, then 
\[\norm{\varphi_B - \varphi_{B+x}}_1 \le  200 \rho d.\]
\end{lemma}
\begin{proof}
The assertion of the lemma is equivalent to $|B \triangle (B+x)| \le  200 \rho d |B|$. Note $B+x \subseteq B_{1+\rho}$. 
If $\rho \ge \frac{1}{100 d}$, the lemma is obvious. Otherwise, by the regularity of $B$, we have $|B_{1+\rho}| \le (1+100d \rho)|B|$, which implies the  lemma. 
\end{proof}
We also need the following generalization of \cite[Lemma 12.1]{bloom2021breaking}. 
\begin{lemma}
\label{lem:BohrSet_Uniform}
There exists a constant $c>0$ such that the following is true. Let $B$ be a regular Bohr set of rank $d$, and $A \subseteq B$ with $\mu_{B}(A)=\alpha$. Let $\gamma>0$, and suppose $\cB $ is a finite family of Bohr sets where each $B'\in \cB$ is a subset of $B_\rho$, where $\rho \le  \frac{c \alpha \gamma}{d|\cB|}$. Then, there exists $x\in G$ such that one of the following two cases holds for the shift $A'=A+x$.
\begin{enumerate}
\item For every $B'\in \cB$, we have $|\mu_{B'}(A')-\alpha|\leq \gamma \alpha$. 

\item  There exists $B'\in \cB$ such that $\mu_{B'}(A') \ge \left(1+\frac{\gamma}{2|\cB|}\right) \alpha$.
\end{enumerate}
\end{lemma}
\begin{proof}
If $c>0$ is small enough, then by \Cref{lem:BohrSetSumUniform}, for every $B' \in \cB$,
\[\left|\langle A*\varphi_{B'}, \varphi_B\rangle - \langle A, \varphi_B\rangle \right|=\left|\langle A , \varphi_{B'} \star \varphi_B\rangle - \langle A, \varphi_B\rangle \right|  \leq \|\varphi_{B'} \star \varphi_{B}- \varphi_B\|_1 \leq \frac{\alpha \gamma}{2|\cB|}. \]
In particular, since $\langle A, \varphi_B\rangle = \alpha$, we have 
\[\langle A*\varphi_{B'}, \varphi_B\rangle \geq \left(1-\frac{\gamma}{2|\cB|}\right)\alpha, \]
which, by averaging over all $B'\in \cB$, implies
\[\left(1-\frac{\gamma}{2|\cB|}\right)\alpha\leq \Ex_{\bfx \in B}\Ex_{B'\in \cB} A*\varphi_{B'}(\bfx)=
\Ex_{\bfx \in B}\Ex_{B'\in \cB} \mu_{B'}(\bfx-A) =\Ex_{\bfx \in B}\Ex_{B'\in \cB} \mu_{B'}(A-\bfx) .\]
Here, the last equality follows from the fact that $B' = -B'$ for any Bohr set $B'$.
Let $x\in B$ be such that 
\begin{equation}\label{eq:avgbohr}
    \Ex_{B'\in \cB} \mu_{B'}(A-x)\geq \left(1-\frac{\gamma}{2|\cB|}\right)\alpha.
\end{equation}
Suppose there is $B''\in \cB$ such that $\mu_{B''}(A-x)\leq (1-\gamma)\alpha$. Combining this with \eqref{eq:avgbohr}, implies that there is $B'\in \cB$ such that $\mu_{B'}(A-x)\geq \left(1+\frac{\gamma}{2|\cB|-2}\right)\alpha$ which implies item (2).
\end{proof}
 
We finally give the definition of an algebraically spread set with respect to Bohr sets.

\begin{definition}[Algebraic Spreadness for Bohr Sets]\label{def:bohralg_spread} Let $B\subseteq G$ be a regular Bohr set and $A \subseteq B$ with $\mu_B(A)=\alpha$. We say that $A$ is a \emph{$(\delta,d',r)$-algebraically spread} subset of $B$ if for every $x_0 \in G$ and every regular Bohr set $B'$ with $\rank(B') \le \rank(B)+d'$ and $|B'| \ge 2^{-r} |B|$, we have 
\(\mu_{B'}(A+x_0)\le (1+\delta) \alpha. \)
\end{definition}

\subsection{Counting lemma over abelian groups}
In this section, we prove an algebraic counting lemma that will provide the key density increment step in the proof of \Cref{thm:L-freeAbelian}. Roughly speaking, this result shows that an algebraically spread set must contain many instances of a given binary system of linear forms.  
 
\begin{theorem}[Algebraic counting lemma]\label{theorem:AlgCountingGeneral}
    Let $\cL = \{L_1,\ldots,L_m\}$ be a binary system of linear forms over variables $x_1,\ldots,x_k$. For every $\eps>0$, there is a $\delta = \delta(\eps, \cL)>0$ such that the following is true. 
    
    Suppose $G$ is a finite abelian group such that $|G|$ is coprime with all the coefficients in $\cL$. Let $A \subseteq B$ have density $\mu_B(A) \ge \alpha$ where $B\subseteq G$ is a regular Bohr set of rank $d=O(1/\alpha)$. If $A$ is $(\delta,d',r)$-algebraically spread, where $r = O_\delta(d\log^{2}(2/\alpha)  + \log^{14}(2/\alpha))$ and $d' = O_\delta(\log^{12}(2/\alpha)) $ are sufficiently large, then there exist Bohr sets $B^{(1)},\ldots,B^{(k)}$ and a shift $A'=A+x_0$ such that the following statements hold.
\begin{itemize}
\item For every $i \in [k]$,  we have $|B^{(i)}| \ge 2^{-O_{\cL,\eps}(d \log^2(2/\alpha))}|B|$; 
\item We have
\begin{equation}\label{eq:LinearSysCountAlgCountLemma}
\Ex_{\bfx_1\in B^{(1)},\ldots,\bfx_k\in B^{(k)}} \prod_{i=1}^m A'(L_i(\bfx_1,\ldots,\bfx_k )) \ge (1-\eps) \alpha^m.
\end{equation}
\end{itemize}
\end{theorem}

To prove \cref{theorem:AlgCountingGeneral}, we will need an almost-periodicity theorem from \cite{BS23improvment}. The following theorem is stated with $U_1=U_2$ in \cite{BS23improvment}, but the same proof works for arbitrary $U_1$ and $U_2$.

\begin{theorem}[{Almost Periodicity \cite[Lemma 8]{BS23improvment}}]
\label{thm:almost_periodicity}
   Let  \(\eps \in (0, \frac{1}{10})\), and let \(B, B', B'' \subseteq G\) be regular Bohr sets of rank \(d\). Suppose that \(U_1, U_2 \subseteq B\), \(A_1 \subseteq B'\), and \(A_2 \subseteq B''\)  with densities  $\mu_B(U_1),\mu_B(U_2) \ge \tau$, $\mu_{B'}(A_1)=\alpha_1$ and $\mu_{B''}(A_2)=\alpha_2$. 
 Let \(\Gamma \subseteq G\) be a set with \(|\Gamma| \leq 2 |B''|\) such that
\begin{enumerate}
    \item \(\Ex_{\bfx_1\in A_1,\bfx_2\in A_2} \Gamma(\bfx_1-\bfx_2)\geq 1-\eps\)
    \item  For all $x\in \Gamma$, \(
   \Ex_{\bfz \in B}U_1(x+\bfz)U_2(\bfz) \ge \left(1+2\eps\right) \mu_B(U_1)\mu_B(U_2) \)
\end{enumerate}
Let $C = \log(2/\tau)+\log(d)+\log\log(2/\alpha_1)+\log\log(2/\alpha_2)$.
There exists $x_0 \in G$ and a regular Bohr set \(B^*\subseteq B'\) with
\[\rank(B^*) \leq d + O_\eps\left( \log^2(2/\tau)\log(2/\alpha_1) \log(2/\alpha_2)\right) \]
and   
\[|B^*|\geq 2^{-O_\eps(r)} |B'|\]
where 
\[r=C \cdot \left(d+\log^2(2/\tau)\log(2/\alpha_1) \log(2/\alpha_2)\right)\]
such that 
   \[\mu_{B^*}(U_1+x_0)\geq \left(1+\frac{\eps}{4}\right)\mu_{B}(U_1).\] 
\end{theorem}

We will also need the following theorem, whose proof uses one of the key ideas in the work of Kelley and Meka~\cite{KM23}, namely, the dependent random choice technique. 

\begin{theorem}[Dependent Random Choice]
\label{thm:dependent}
Let $G$ be a finite abelian group, let $\cX_1,\cX_2,\cY \subseteq G$, and $\tau>0$. Suppose $A \subseteq \cX_1$ and $B \subseteq \cX_2$  satisfy 
\[\Ex_{(\bfx_1,\bfx_2) \in \cX_1 \times \cX_2} \left(\Ex_{\bfy \in \cY}A(\bfx_1+\bfy)B(\bfx_2+\bfy)\right)^p  \ge (1+2\eps)^p \tau^p,\]
where $\eps>0$ and $p \ge \Omega(\log(2/\eps)/\eps)$. Then there exist $A_1\subseteq \cX_1$ and $A_2 \subseteq \cX_2$ with $\mu_{\cX_1}(A_1)\mu_{\cX_2}(A_2) \ge \tau^p$ such that 
\[\Pr_{(\bfx_1,\bfx_2) \in A_1 \times A_2}\left[ \Ex_{\bfy \in \cY}A(\bfx_1+\bfy)B(\bfx_2+\bfy) >\left(1+\eps\right) \tau \right] \ge 1 - \frac{\eps}{100}. \]
\end{theorem}
\begin{proof}
Our assumption says 
\[\Ex_{\substack{(\bfx_1,\bfx_2) \in \cX_1 \times \cX_2\\ \bfy_1,\ldots,\bfy_p \in \cY}} \prod_{j=1}^p A(\bfx_1+\bfy_j)B(\bfx_2+\bfy_j)  \ge (1+2\eps)^p \tau^p.\]
Denote $\vec{x}=(x_1,x_2)$ and $\vec{y}=(y_1,\ldots,y_p)$, and let
\[\mathcal{F}(\vec{x},\vec{y}) \defeq  \prod_{j=1}^p A(x_1+y_j)B(x_2+y_j).\]
Define 
\[A'_{\vec{y}}\defeq  \{x \in \cX_1 : x+y_1,\dots,x+y_k \in A\} =  \bigcap_{i=1}^k (A-y_i),\]
and denote $\alpha_{\vec{y}} \defeq \mu_{\cX_1}(A'_{\vec{y}})$. Similarly, define 
\[B'_{\vec{y}}\defeq  \{x \in \cX_2 : x+y_1,\dots,x+y_k \in B\} = \bigcap_{i=1}^k (B-y_i),\]
and denote $\beta_{\vec{y}} \defeq \mu_{\cX_2}(B'_{\vec{y}})$.  We have 
\begin{align*}
(1+2\eps)^p \tau^p &\le \Ex_{\vec{\bfx},\vec{\bfy}} \mathcal{F}(\vec{x},\vec{y}) =  \Ex_{\vec{\bfy}} \left(\Ex_{\bfx_1  \in \cX_1} \prod_{j=1}^p A(\bfx_1+\bfy_j)\right)\left(\Ex_{\bfx_2 \in \cX_2} \prod_{j=1}^p B(\bfx_2+\bfy_j)\right) \\  &=\Ex_{\vec{\bfy}} [\alpha_{\vec{\bfy}}\beta_{\vec{\bfy}}].
\end{align*}
We will only be interested in those $\vec{y}$ where $A'_{\vec{y}}$ and $B'_{\vec{y}}$ are not very small. Let  
\[\Lambda = \left\{\vec{y} \in \cY^p : \alpha_{\vec{y}}\beta_{\vec{y}}> \tau^p  \right\}.\]
We have 
\begin{equation}
\label{eq:appendix1}
\Ex_{\vec{\bfy}} \left[ \Lambda(\vec{\bfy}) \Ex_{\vec{\bfx}} \left[\mathcal{F}(\vec{\bfx},\vec{\bfy})\right]\right] = \Ex_{\vec{\bfy}} \Lambda(\vec{\bfy}) \alpha_{\vec{\bfy}} \beta_{\vec{\bfy}}  \ge  (1+2\eps)^p \tau^p -  \tau^p= \left((1+2\eps)^p-1\right)\tau^p.  
\end{equation}

Let
\[S \defeq \left\{(x_1,x_2) \in \cX_1 \times \cX_2 : \Ex_{\bfy \in \cY}A(x_1+\bfy)B(x_2+\bfy) \le \left(1+\eps\right) \tau \right\}.\]
Note that 
\[ \Ex_{\vec{\bfx},\vec{\bfy}} \mathcal{F}(\vec{\bfx},\vec{\bfy}) = \Ex_{(\bfx_1,\bfx_2) \in \cX_1 \times \cX_2} \left(\Ex_{\bfy \in \cY}A(\bfx_1+\bfy)B(\bfx_2+\bfy)\right)^p, \]
and 
\begin{equation}
\label{eq:appendix2}
\Ex_{\vec{\bfx},\vec{\bfy}} \mathcal{F}(\vec{\bfx},\vec{\bfy})  S(\vec{\bfx})  \le \left(1+ \eps\right)^{p} \tau^p.
\end{equation}

Putting \eqref{eq:appendix1} and \eqref{eq:appendix2} together gives
\[
\frac{\Ex_{\vec{\bfy}}  \Ex_{\vec{\bfx}} \mathcal{F}(\vec{\bfx},\vec{\bfy}) S(\vec{\bfx})}{\Ex_{\vec{\bfy}} \Lambda(\vec{\bfy})\Ex_{\vec{\bfx}} \mathcal{F}(\vec{\bfx},\vec{\bfy})} \le \frac{(1+\eps)^p}{(1+2\eps)^p-1}.
\]
In particular, for $p=\Omega(\log(2/\eps)/\eps)$, there exists a choice of $\vec{y} \in \Lambda$ such that 
\[
\frac{ \Ex_{\vec{\bfx}} \mathcal{F}(\vec{\bfx},\vec{y}) S(\vec{\bfx})}{\Ex_{\vec{\bfx}} \mathcal{F}(\vec{\bfx},\vec{y})} \le \frac{(1+\eps )^p}{(1+2\eps)^p-1} \le \frac{\eps}{100}.
\]
Now note that $\mathcal{F}(\vec{x},\vec{y})=1$ iff $(x_1,x_2) \in A'_{\vec{y}} \times B'_{\vec{y}}$, and thus
\[
\Pr_{(\bfx_1,\bfx_2) \in A'_{\vec{y}} \times B'_{\vec{y}}} [(\bfx_1,\bfx_2) \in S]=\frac{ \Ex_{\vec{\bfx}} \mathcal{F}(\vec{\bfx},\vec{y}) S(\vec{\bfx})}{\Ex_{\vec{\bfx}} \mathcal{F}(\vec{\bfx},\vec{y})}  \le \frac{\eps}{100}.
\]
We set $A_1 \defeq A'_{\vec{y}}$ and $A_2 \defeq B'_{\vec{y}}$. Since $\vec{y} \in \Lambda$, we have $\mu_{\cX_1}(A_1)\mu_{\cX_2}(A_2) \ge \tau^p$, as desired.  
\end{proof}

\subsubsection{Proof of \texorpdfstring{\Cref{theorem:AlgCountingGeneral}}{Theorem \ref{theorem:AlgCountingGeneral}}}
Suppose the underlying graph of $\cL$ is an oriented graph $H = ([k],E)$ such that $(i,j)\in E$ implies $i<j$. For convenience, we re-index $\cL =  \{L_{ij}: (i,j) \in E\}$  where $L_{ij}(x_1,\ldots,x_k) = \lambda_{ij}x_{i}+ \eta_{ij} x_{j}$ with $i< j$. By our assumption, all the coefficients are coprime with $|G|$. 
Let $\gamma =\Omega_{m,\eps}(1)$ be a small constant to be determined. We emphasize that $\gamma$ does not depend on the particular choice of the coefficients in $\cL$. Let 
\begin{equation}
\label{eq:rho}
K \defeq \prod_{(i,j) \in E} |\eta_{ij}||\lambda_{ij}| \ \ \text{and} \ \ \rho \defeq \frac{\gamma \exp(-\gamma^{-1} \log^2(2/\alpha)) }{10^3 d K^{k+1}}.
\end{equation} 
By \Cref{lem:BohrRegular}, there exist $\frac{\rho^k}{2^{k} K^{k^2}}\le  \rho_1 \le \cdots \le \rho_k \le \frac{\rho}{K}$ such that the following properties hold. 
\begin{itemize}
\item[{\bf (P1)}] $B_{\rho_i}$'s are regular for all $i$; 
\item[{\bf (P2)}] $ \frac{\rho \rho_{i+1}}{2K^{k}} \le \rho_i \le  \frac{\rho \rho_{i+1}}{K^{k}}$ for all $i=1,\ldots,k-1$. In particular, $\frac{\rho^k}{2^{k} K^{k^2}}\le  \rho_1$.
\end{itemize}
Let $B^{(i)}= K^{k-i} \cdot B_{\rho_i}$ for all $i$. 
Note that the conditions on $\rho$ and $\rho_i$ guarantee the following. 
\begin{itemize}
\item[{\bf (P3)}] for all $(i,j) \in E$, we have 
\[\lambda_{ij} \cdot B^{(i)} = \lambda_{ij}K^{k-i} \cdot B_{\rho_i} \subseteq B_{ \rho_k K} \subseteq B_\rho \text{ and } \eta_{ij}\cdot B^{(j)} =\eta_{ij}K^{k-j}\cdot B_{\rho_j} \subseteq B_{\rho_k K} \subseteq  B_\rho;\]
\item[{\bf (P4)}] For every $(i,j)\in E$, since $i<j$ and $K$ is divisible by $\eta_{ij}$, we have $\frac{\lambda_{ij} K^{j-i}}{\eta_{ij}} \cdot B_{\rho_i} \subseteq B_{\rho \cdot \rho_j}$.
In particular, for every $(i,j)\in E$, we have 
\[ \lambda_{ij} \cdot B^{(i)} \subseteq B'_\rho \ \ \text{with} \ \ B' \defeq \eta_{ij} \cdot B^{(j)};\]  
 
\item[{\bf (P5)}] By \Cref{lem:Bohr_Size}, for all $i$, we have 
\[ |B^{(i)}| \ge (\rho_i/4)^d |B| \ge \left(\frac{\rho^k}{4 \times 2^{k} K^{k^2}}\right)^d |B| \ge 2^{-O_{\cL,\eps}(d \log^2(2/\alpha))}|B|,\]
where the last inequality uses the assumption $d=O(1/\alpha)$.  
\end{itemize}
Consider the collection $\cB = \{\lambda_{ij}\cdot B^{(i)} : (i,j)\in E\} \cup \{\eta_{ij}\cdot B^{(j)}: (i,j)\in E\}$ of Bohr sets and observe that $|\cB|\leq 2m$.  Note that by {\bf (P3)}, all these Bohr sets are subsets of $B_{\rho}$, and moreover, if $\gamma \leq O(1/m)$ is sufficiently small, we can ensure $\rho<\frac{c \alpha \gamma}{d|\cB|}$, where $c$ is the constant from \Cref{lem:BohrSet_Uniform}. Therefore, we can apply \Cref{lem:BohrSet_Uniform} with $A$, $\cB$, $d$, $\gamma$, $\alpha_*=\mu_B(A)\ge \alpha$ to guarantee that there is a shift $A' = A+x_0 $ such that one of the following two cases holds: 
\begin{enumerate}
    \item[{\bf (P6)}] 
    For every $B'\in \cB$, 
    \begin{equation}\label{eq:uniformOnBohrFamily}
    |\mu_{B'}(A')-\alpha_*|\leq \gamma \alpha_*;
    \end{equation}
    \item[{\bf (P$\mathbf{6'}$)}] There is $B'\in \cB$ such that $\mu_{B'}(A')\geq \left(1+\frac{\gamma}{2|\cB|}\right)\alpha_*\geq  \left(1+\frac{\gamma}{4m}\right)\alpha_*$.
\end{enumerate}
Since $A$ is $(\delta,d',r)$-algebraically spread, assuming $\delta < \frac{\gamma}{8m}$ and $r = \Omega_{m,\eps}(d \log^2(2/\alpha))$ is sufficiently small, {\bf (P$\mathbf{6'}$)} cannot hold, and therefore, \eqref{eq:uniformOnBohrFamily} holds. 
Assume to the contrary of the theorem that \eqref{eq:LinearSysCountAlgCountLemma} does not hold. In particular, 
\begin{equation}\label{eq:LinearSysCountAlgCountLemma_contra}
\left| \Ex_{\bfx_1\in B^{(1)},\ldots,\bfx_k\in B^{(k)}} \prod_{i=1}^m A'(L_i(\bfx_1,\ldots,\bfx_k )) - \alpha_*^m\right| > 
\eps \alpha_*^m.
\end{equation}
We will show that $A$ is not $(\delta,d',r)$-algebraically spread, reaching a contradiction.
 We divide the proof into two stages. In the first part, we find some large subsets $S, T$ such that 
\[\Ex_{\bfx\in  S,\bfy \in T} A'(\bfx+\bfy)\geq (1+100\delta)\alpha_*.\] 
In the second stage, we obtain a Bohr set ${B}^*$ of rank $d+d'$ and relative density at least $2^{-r}$ such that \[\mu_{{B}^*}(A')\geq (1+\delta)\alpha_*\]
which gives the desired contradiction.

\paragraph{Stage 1: Density increment on a rectangle.} 

For every $(i,j) \in E$, define $G_{ij}  \subseteq B^{(i)} \times B^{(j)}$ as $G_{ij}(x,y) \defeq A'(\lambda_{ij}x+\eta_{ij} y)$ and let  $\alpha_{ij}$ denote the density of $G_{ij}$ in $B^{(i)} \times B^{(j)}$.

\begin{claim}
\label{claim:regular_densities}
For any $(i,j)\in E$ and every $x \in B^{(i)}$, we have 
\begin{equation}
\label{eq:almost_regular_bohr}
 |\Ex_{\bfy \in B^{(j)}} G_{ij}(x, \bfy) -\alpha_*| \le 2\gamma \alpha_*.
\end{equation}
In particular,      
\begin{equation*}
       \left|\alpha_{ij}- \alpha_*\right|\leq 2\gamma \alpha_*, 
\end{equation*}
and 
\begin{equation}\label{eq:alphaProductsClose}
\left|\prod_{(i,j)\in E} \alpha_{ij}-\alpha_*^m\right|= O(m \gamma \alpha_*^m).
\end{equation}
\end{claim}
\begin{proof}
Denote $B' = \eta_{ij} \cdot B^{(j)}$ and recall that by {\bf (P4)}, $\lambda_{ij} x \in \lambda_{ij} \cdot B^{(i)} \subseteq B'_\rho$. Therefore, by \Cref{lem:BohrSetSumUniform} and the choice of $\rho$, we have
\[\norm{\varphi_{B'+\lambda_{ij}x} -\varphi_{B'}}_1 \leq 200 \rho d \le \gamma \alpha.\]
In particular, 
\[ |\Ex_{\bfy \in B'} A'(\lambda_{ij}x+ \bfy) -\mu_{B'}(A')| =  |\inp{\varphi_{B'+\lambda_{ij}x} -\varphi_{B'}}{A'}|  \le \gamma \alpha.\]
By \eqref{eq:uniformOnBohrFamily}, we have $|\mu_{B'}(A')-\alpha_*| \le \gamma \alpha_*$, and therefore, we can conclude 
\[ |\Ex_{\bfy \in B^{(j)}} G_{ij}(x, \bfy) -\alpha_*|= |\Ex_{\bfy \in B'} A'(\lambda_{ij}x+ \bfy) -\alpha_*| \le 2\gamma \alpha_*,\]
which verifies \eqref{eq:almost_regular_bohr}. Averaging this inequality over $x \in B^{(i)}$ yields $|\alpha_{ij}-\alpha_*| \le 2 \gamma \alpha_*$. Finally, \eqref{eq:alphaProductsClose} follows from 
\[(1-2\gamma)^m \alpha_*^m \le \prod_{(i,j)\in E} \alpha_{ij} \le (1+2\gamma)^m \alpha_*^m. \qedhere\] 
\end{proof}
 
By \eqref{eq:LinearSysCountAlgCountLemma_contra}  and \eqref{eq:alphaProductsClose}, if  $\gamma = O(\eps/m)$ is sufficiently small, we get 
\[\left|\Ex_{\bfx_1\in B^{(1)},\ldots,\bfx_k\in B^{(k)}}  \prod_{(i,j) \in E} G_{ij}(\bfx_i,\bfx_j) - \prod_{(i,j) \in E}\alpha_{ij}\right| \ge \frac{\eps}{2} \prod_{(i,j) \in E}\alpha_{ij}.\]
Now, we are ready to apply our graph counting lemma (i.e., \cref{thm:main_graph_count}). There exists a constant $\delta'=\delta'(\eps,m)>0$ such that for some $(i,j) \in E$, 
at least one of the following cases holds.
\begin{enumerate}
\item[(i)] There exist $S' \subseteq B^{(i)}$ and $T' \subseteq B^{(j)}$  with  $|S'||T'| \ge 2^{-O_{m,\eps}(\log^2(2/\alpha))} |B^{(i)}||B^{(j)}|$ such that  
\begin{equation}
\label{eq:first_alternative_algebraic_counting}
\Ex_{\substack{(\bfx,\bfy) \in S'\times T'}}[G_{ij}(\bfx,\bfy)] \ge (1+\delta') \alpha_{ij}; 
\end{equation}
\item[(ii)] There exists $S \subseteq B^{(i)}$ of density at least $2^{-O_{m,\eps}(\log(2/\alpha))}$ such that  
\begin{equation*}     
    \Ex_{\bfx \in B^{(j)}}[G_{ij}(z,\bfx)] \le (1-\delta') \alpha_{ij} \qquad \forall z \in S;
\end{equation*}
\end{enumerate}
If $\gamma=O(\delta')$ is sufficiently small, by \Cref{claim:regular_densities}, for every $z \in B^{(i)}$, we have 
\[|\Ex_{\bfx \in B^{(j)}}[G_{ij}(z,\bfx)] - \alpha_{ij}|< \delta' \alpha_{ij},\]
and (ii) cannot hold. Therefore, (i) must be true. In particular, if $\gamma\leq \delta'/8$, 
\begin{equation}
\label{eq:rect_increase_density}
\Ex_{\substack{(\bfx,\bfy) \in S'\times T'}}[G_{ij}(\bfx,\bfy)] \ge (1+\delta') \alpha_{ij} \ge (1+\delta')(1-2\gamma) \alpha_* \ge \left(1+\frac{\delta'}{2}\right) \alpha_*; 
\end{equation}
Denoting 
\[S \defeq \lambda_{ij} \cdot S', \ \ T \defeq \eta_{ij} \cdot T',\ \  B' \defeq \eta_{ij}  \cdot B^{(j)}, \ \ B'' \defeq \lambda_{ij}  \cdot B^{(i)},\]
we have the following: 
\begin{itemize}
\item  $B',B'' \subseteq B$ are regular Bohr sets of rank $d$ that by {\bf (P5)} satisfy 
\[|B'|,|B''| \ge 2^{-O_{\cL,\eps}(d \log^2(2/\alpha))}|B|.\] 
Moreover, by {\bf (P4)}, we have  
$B'' \subseteq B'_\rho$ and $B' \subseteq B_{\rho}$.

\item  By {\bf (P6)}, we have  $|\mu_{B'}(A')-\alpha_*|\leq \gamma \alpha_*$ and  $|\mu_{B''}(A')-\alpha_*|\leq \gamma \alpha_*$.

\item  $T \subseteq B'$ and $S \subseteq B''$ and they satisfy  
\begin{equation}
    \label{eq:densityIncrementOnRectangle}
    \Ex_{\bfs\in S,\bft \in T} A'(\bfs+\bft)\geq \left(1+\frac{\delta'}{2}\right)\alpha_*,
\end{equation}
with $\delta'=\Omega_{m,\eps}(1)$ and $\mu_{B''}(S),\mu_{B'}(T) \ge   2^{-O_{m,\eps}(\log^2(2/\alpha))}$.   
\end{itemize}

\paragraph{Stage 2: Density increment on a Bohr set.}  
We can rewrite \eqref{eq:densityIncrementOnRectangle}
as
\begin{equation}
\label{eq:pre_change_of_variables}
\Ex_{\bfb'' \in B'',\bfb' \in B'} A'(\bfb''+\bfb')S(\bfb'')T(\bfb')\geq \left(1+\frac{\delta'}{2}\right)\alpha_*\cdot \mu_{B''}(S)\cdot  \mu_{B'}(T) \ge  2^{-O_{m,\eps}(\log^2(2/\alpha))}.
\end{equation}
We wish to apply a change of variables and then apply H\"older's inequality to eliminate the set $S$. Since the averages are over Bohr sets rather than the entire group, we will incur some small errors as a result of the change of variables. By  \Cref{lem:BohrSetSumUniform} and our choice of $\rho$ in \eqref{eq:rho}, if $\gamma=\Omega_{\eps,m}(1)$ sufficiently small, for every $z \in B_\rho'' \subseteq B'_\rho$,  
\begin{multline*}
\left|\Ex_{\bfb'' \in B''-z} \Ex_{\bfb' \in B'+z} A'(\bfb''+\bfb')S(\bfb'')T(\bfb')- \Ex_{\bfb'' \in B'',\bfb' \in B'} A'(\bfb''+\bfb')S(\bfb'')T(\bfb')\right|  \\ \le 400 \rho d \leq \gamma  \alpha \cdot \mu_{B''}(S)\cdot  \mu_{B'}(T).
\end{multline*}
By averaging over $z \in B_{\rho}''$ and combining it with \eqref{eq:pre_change_of_variables}, if $\gamma\leq \delta'/8$, we have 
\begin{equation}
\label{eq:preHolder}
\Ex_{\bfz \in B''_\rho,\bfb'' \in B'',\bfb' \in B'} A'(\bfb''+\bfb')S(\bfb''-\bfz)T(\bfz+\bfb') \ge \left(1+ \frac{\delta'}{4}\right)\alpha_* \cdot \mu_{B''}(S)\cdot  \mu_{B'}(T).
\end{equation}
Next, we will apply H\"older's inequality to remove $S$. For any $p \ge 2$ and $q \in (1,2]$ with $\frac{1}{p}+\frac{1}{q}=1$, we have
\begin{multline*}
\Ex_{\bfz \in B''_\rho,\bfb'' \in B'',\bfb' \in B'} A'(\bfb''+\bfb')S(\bfb''-\bfz)T(\bfz+\bfb') \le \\ \left(\Ex_{\bfz \in B''_\rho, \bfb'' \in B''} \left(\Ex_{\bfb' \in B'} A'(\bfb''+\bfb')T(\bfz+\bfb')\right)^p \right)^{1/p} \left(\Ex_{\bfz \in B''_\rho,\bfb'' \in B''} 
S(\bfb''-\bfz)\right)^{1/q}.
\end{multline*}
By \Cref{lem:BohrSetSumUniform}, we have 
\[|\Ex_{\bfz \in B''_\rho,\bfb'' \in B''} 
S(\bfb''-\bfz)  - \mu_{B''}(S)| \le 200d\rho \leq \gamma\mu_{B''}(S),  \] 
which with \eqref{eq:preHolder} gives us 
\[\left(\Ex_{\bfz \in B''_\rho, \bfb'' \in B''} \left(\Ex_{\bfb' \in B'} A'(\bfb''+\bfb')T(\bfz+\bfb')\right)^p \right)^{1/p} 
 \left(1+ \gamma \right)^{1/q} \ge \left(1+\frac{\delta'}{4}\right)\alpha_* \cdot \mu_{B'}(T)\cdot  \mu_{B''}(S)^{1-\frac{1}{q}}.\]
 By \eqref{eq:uniformOnBohrFamily}, we have  $\alpha_*>(1-\gamma)\mu_{B'}(A')$. If $\gamma \leq \delta'/100$, we have
 \[\Ex_{\bfz \in B''_\rho, \bfb'' \in B''} \left(\Ex_{\bfb' \in B'} A'(\bfb''+\bfb')T(\bfz+\bfb')\right)^p  
  \ge \left(1+\frac{\delta'}{8}\right)^p \mu_{B'}(A')^p \mu_{B'}(T)^p\cdot  \mu_{B''}(S).\]
 Taking  $p = O(\log(2\mu_{B''}(S)^{-1})/\delta')=O(\log^2(2/\alpha))$ to be sufficiently large, we obtain that
\begin{equation}
\label{eq:mixed_grid_norm}
 \Ex_{\bfz \in B''_\rho, \bfb'' \in B''} \left(\Ex_{\bfb' \in B'} A'(\bfb''+\bfb')T(\bfz+\bfb')\right)^p     \ge \left(1+\frac{\delta'}{16} \right)^p \mu_{B'}(A')^p   \mu_{B'}(T)^p.
\end{equation}

Now, we are in the position to apply \Cref{thm:dependent} with $\cX_1 = B''_\rho$, $\cX_2=B''$, and $\cY=B'$. Therefore, there exists  sets $A_1 \subseteq B''_\rho$ and $A_2 \subseteq B''$, and $\delta''=\frac{\delta'}{32}$, such that 
\begin{equation}
\label{eq:densities_A1_A2}
\mu_{B_\rho''}(A_1)\mu_{B''}(A_2) \ge \mu_{B'}(A')^p \mu_{B'}(T)^p=2^{-O_{\eps,m}(p\log(2/\alpha) + p\log^2(2/\alpha))}=2^{-O_{\eps,m}(\log^4(2/\alpha))}.
\end{equation}
and
\begin{equation}
\label{eq:after_dependent}
\Pr_{(\bfx_1,\bfx_2) \in A_1 \times A_2}\left[ \Ex_{\bfz \in B'}A'(\bfx_1+\bfz)T(\bfx_2+\bfz) >\left(1+ \delta'' \right) \mu_{B'}(A') \mu_{B'}(T)\right] \ge 1 - \frac{\delta''}{100}. 
\end{equation}
Note that $(x_1,x_2) \in A_1 \times A_2$ implies $x_2 \in B'' \subseteq B'_{\rho}$ where the latter inclusion is by {\bf (P4)}.  Therefore, by \Cref{lem:BohrSetSumUniform},  assuming $\gamma \leq \delta''/10$,  for every $(x_1,x_2) \in A_1 \times A_2$, we have 
\begin{equation}
\label{eq:after_dependent_shift}
|\Ex_{\bfz \in B'}A'(x_1+\bfz)T(x_2+\bfz)-\Ex_{\bfz \in B'}A'(x_1-x_2+\bfz)T(\bfz)| \le 200 \rho d  \le \frac{\delta''}{10} \mu_{B'}(A') \mu_{B'}(T). 
\end{equation}
Let 
\[ \Gamma = \{b \in B'' : \Ex_{\bfz \in B'}A'(b+\bfz)T(\bfz) \ge \left(1+0.9\delta''\right) \mu_{B'}(A') \mu_{B'}(T) \}.\]
 By \eqref{eq:after_dependent} and \eqref{eq:after_dependent_shift}, we have 
\[\Ex_{(\bfx_1,\bfx_2) \in A_1 \times A_2}\left[\bfx_1-\bfx_2 \in \Gamma \right] \ge 1 - \frac{\delta''}{100}. \]
By \Cref{lem:BohrRegular}, we may assume that $B'''\coloneqq B''_\rho$ is regular. 
We invoke \Cref{thm:almost_periodicity} with $2\eps = 0.9\delta''$, $\tau=2^{-O_{\eps,m}(\log^2(2/\alpha))}$, and $A',T\subseteq B',A_1\subseteq B'''$ and $A_2\subseteq B''$. Note that $|\Gamma|\leq |B''|$, and  we have  
\begin{itemize}
\item $\mu_{B'}(A')\geq (1-\gamma)\alpha \ge \tau$;
\item $\mu_{B'}(T)\geq 2^{-O_{\eps,m}(\log^2(2/\alpha))} \ge \tau$;
\item $\alpha_1\defeq \mu_{B'''}(A_1)\geq 2^{-O_{\eps,m}(\log^4(2/\alpha))}$ and $\alpha_2 \defeq \mu_{B''}(A_2)\geq 2^{-O_{\eps,m}(\log^4(2/\alpha))}$  by \eqref{eq:densities_A1_A2}.
\end{itemize}    
 We obtain a regular Bohr set $B^*$ and $x\in G$ such that 
\[\frac{|(A'+x)\cap B^*|}{|B^*|}\geq \left(1+\frac{\delta''}{10}\right)\frac{|A'|}{|B'|}\geq \left(1+\frac{\delta''}{10}\right)\left(1-\gamma\right)\frac{|A|}{|B|}\geq \left(1+\frac{\delta''}{20}\right)\frac{|A|}{|B|}.\]
The last inequality is by assuming $\gamma\leq \delta''/40$. Moreover, 
\begin{align*}
\rank(B^*)&\leq \rank(B')+O_{\eps}(\log^2(2\tau^{-1})\log(2\alpha_1^{-1})\log(2\alpha_2^{-1}))\\
&\leq \rank(B')+O_{\eps}\left(\log^{12}(2\alpha^{-1})\right)
\end{align*}
and \[|B^*|\geq 2^{-O_{m,\eps}(r)} |B'''|,\]
where
\[C = \log(2/\tau)+\log(\rank(B'))+\log\log(2/\alpha_1)+\log\log(2/\alpha_2)=O(\log(2/\tau))=O(\log^2(2/\alpha))\]
and
\begin{align*}
r&\leq  
C\cdot
(
\rank(B')+\log^2(2\tau^{-1})\log(2\alpha_1^{-1})\log(2\alpha_2^{-1})
)\\
&\leq
C\cdot
\left(\rank(B') + \log^{12}(2\alpha^{-1})\right)\\
&\leq O\left( \log^2(2\alpha^{-1})d + \log^{14}\left(2\alpha^{-1}\right)\right).
\end{align*}
Combining $|B^*|\geq 2^{-O_{m,\eps}(r)} |B'''|$ with 
\[|B'''|=|B''_\rho| \geq \left(\frac{\rho}{4}\right)^d |B''|\geq \exp\left(-O_{m,\eps}(d\log^2{(2\alpha^{-1})})\right)|B|,\]
we have 
\[|B^*| \ge \exp\left(- O_{m,\eps}\left( d\log^2(2\alpha^{-1}) + \log^{14}(2\alpha^{-1})\right)\right)|B|.\]

By taking $\gamma = \Omega_{m,\eps}(1)$  to be a small enough constant, we can satisfy all the requirements on $\gamma$.
Taking $\delta=\min\left(\frac{\delta''}{20},\frac{\gamma}{10m}\right)$ finishes the proof of \Cref{theorem:AlgCountingGeneral}.
\subsection{Proof of \texorpdfstring{\Cref{thm:L-freeAbelian}}{Theorem \ref{thm:L-freeAbelian}} (i)}

Let $\cL = \{L_1,\ldots,L_m\}$ and $\alpha=\mu(A)$. Suppose $d' =  O_{\cL}(\log^{12}(2/\alpha))$ and $r=O_{\cL}(\log^{15}(2/\alpha))$ are sufficiently large. 
Let $\delta = \delta(\eps, \cL)>0$ given by \Cref{theorem:AlgCountingGeneral}  with $\eps = 1/2$. We wish to obtain an algebraically spread set $A_*$ to apply \Cref{theorem:AlgCountingGeneral}. If  $A$ is $(\delta,d',r)$-algebraically spread then $A_*=A$, otherwise, from  \Cref{def:bohralg_spread} (where $B=G$),  we find a regular Bohr set $B^{[1]}\subseteq G$ of rank $d'$ and density at least $2^{-r}$ and an $x_1\in G$ such that $A_1 \coloneqq (A+x_1)\cap B^{[1]}$ satisfies $\mu_{B^{[1]}}(A_1)\geq (1+\delta)\alpha$. We can repeat this density increment process at most $O_\delta(\log(2/\alpha))$ times and obtain a regular Bohr set $B^*$ and $x_* \in G$ such that $A_*\coloneqq (A+x_*)\cap B^*$ is $(\delta,d',r)$-algebraically spread in $B^*$.
At this point, we have 
\[\rank(B^*)\leq O_{\cL}(d'\log(2/\alpha)) =  O_{\cL}(\log^{13}(2/\alpha))\] 
and for 
\[\mu(B^*)\geq \exp\left(- O_{\cL}(\log^{16}(2/\alpha))\right).\]
Note $\alpha_*\defeq \mu_{B^*}(A_*)\geq\alpha$.
Since $r = O_\delta(\rank(B^*)\log^{2}(2/\alpha)  + \log^{14}(2/\alpha))=O_{\cL}(\log^{15}(2/\alpha))$ is sufficiently large, by \Cref{theorem:AlgCountingGeneral}, there are Bohr sets $B^{(1)},\cdots, B^{(k)}$ such that 
\begin{itemize}
\item For every $i \in [k]$,  we have 
\begin{equation}\label{eq:bohrlowerdensity}
    \mu(B^{(i)}) \ge 2^{-O_{\cL}( \log^{15}(2/\alpha))}\mu(B^*) \ge 2^{-O_{\cL}( \log^{16}(2/\alpha))}.
\end{equation}
\item We have 
\begin{equation}\label{eq:lowerallprob}
    \Pr_{\bfx_1\in B^{(1)},\ldots,\bfx_k\in B^{(k)}}[\cL(\bfx_1,\ldots,\bfx_k)\in A_*^m]\geq  \alpha^m/2.
\end{equation}
\end{itemize}
Note that $x\mapsto\lambda_{ij}x$ and $x\mapsto\eta_{ij}x$  are group automorphisms.  Hence, given any two distinct linear forms $L,L'\in \cL$, we have
\begin{equation}
    \Pr_{\bfx_1\in B^{(1)},\ldots,\bfx_k\in B^{(k)}}[L(\bfx_1,\ldots,\bfx_k)=L'(\bfx_1,\ldots,\bfx_k)]\leq \frac{1}{\min_{i}|B^{(i)}|}.
\end{equation}
This implies that 
\begin{equation}\label{eq:upperdegenerate}
    \Pr_{\bfx_1\in B^{(1)},\ldots,\bfx_k\in B^{(k)}}[\cL(\bfx_1,\ldots,\bfx_k) \text{ is degenerate}]\leq \frac{m^2}{\min_{i}|B^{(i)}|}.
\end{equation}
Since we have assumed that $A$ does not contain non-degenerate instances of $\cL$, by combining \eqref{eq:bohrlowerdensity}, \eqref{eq:lowerallprob}, and \eqref{eq:upperdegenerate}, we get    
\begin{equation}\label{eq:finalinequality}
    \frac{\alpha^m}{2}\leq \frac{m^2}{\min_{i}|B^{(i)}|} \leq m^2\cdot|G|^{-1}\cdot \exp(O_{\cL}(\log^{16}(2/\alpha))).
\end{equation}
If $\alpha \geq \exp\left(- O_{\cL}(\log^{1/16}(|G|))\right)$, we can ensure the right hand side in \eqref{eq:finalinequality} is at most, say, $|G|^{-1/2}$, which violates \eqref{eq:finalinequality} if $|G|= \Omega_{\cL}(1)$ is large enough. This is a contradiction.
\qed
 
\subsection{Proof of \texorpdfstring{\Cref{thm:L-freeAbelian}}{Theorem \ref{thm:L-freeAbelian}} (ii)}
 
Recall that we want to leverage the assumption  that the underlying graph of $\cL$ is $2$-degenerate to conclude the stronger bound of 
\[|A| \le |G|\cdot 2^{-\Omega_{\cL}(\log^{1/9}|G|)}. \]

We will describe how the proof of \Cref{theorem:AlgCountingGeneral} can be modified to yield a stronger density increment. We change the definition of $\rho$ in \eqref{eq:rho} to the following, shaving off a factor of $\log(2/\alpha)$.   
\begin{equation}
\label{eq:rho_2deg}
K \defeq \prod_{(i,j) \in E} |\eta_{ij}||\lambda_{ij}| \ \ \text{and} \ \ \rho \defeq \frac{\gamma \exp(-\gamma^{-1} \log(2/\alpha)) }{10^3 d K^{k+1}}.
\end{equation}

Let $H$ be the underlying graph of $\cL$. 
Since $H$ is $2$-degenerate, there exists an ordering of the vertices such that every vertex has at most $2$ preceding neighbours. We then orient each edge of $H$ from the smaller to the larger vertex in this ordering. 
Note that the indegree of every vertex is at most $2$. Therefore, if one applies \Cref{lem:MainTechnical} with this graph $H$, one would obtain  
\[ \norm{G_{ij}}_{U(2,p)} \geq (1+\delta')\alpha_{ij},\]
for $p=O_{m,\eps} (\log \frac{2}{\alpha})$ and  $\delta'=\eps^22^{-4m-10}$. Recalling the definitions of $B'$ and $B''$, we can express this as,
\begin{equation}
\label{eq:Updated}
\Ex_{\bfx \in B'', \bfy \in B''} \Ex_{\bfz \in B'}\left(A'(\bfx+\bfz)A'(\bfy+\bfz)\right)^p \ge  (1+\delta')^{2p}\alpha_{ij}^{2p}. 
\end{equation}
Note that this allows us to bypass obtaining the large rectangle $S \times T$ and directly jump to  \eqref{eq:mixed_grid_norm} with $T=A'$, where we also use $\alpha_{ij}\geq (1-O(\gamma))\mu_{B'}(A')$.  Furthermore, we are in a better situation in terms of the parameters: Firstly, $p=O_{m,\eps} (\log(2/\alpha))$ instead of $p=O_{m,\eps} (\log^2(2/\alpha))$. Secondly, in \eqref{eq:Updated}, we have $\mu_{B'}(A')^{2p}=2^{-O(\log^2(2/\alpha))}$ instead of $\mu_{B'}(A')^p \mu_{B'}(T)^p=  2^{-O_{m,\eps}(p\log(2/\alpha)+p\log^2(2/\alpha))}=2^{-O(\log^4(2/\alpha))}$ in \eqref{eq:mixed_grid_norm}.  Note also that the reason we could choose a larger $\rho$ is that in \eqref{eq:after_dependent_shift},  we now only need  $200\rho d$ to be sufficiently small compared to $\frac{\delta''}{10} \mu_{B'}(A') \mu_{B'}(A')=\Omega(\alpha^2)$. 
Therefore, when we apply \Cref{thm:dependent}, we obtain $A_1,A_2 \subseteq B''$ with 
\[\mu_{B''}(A_1),\mu_{B''}(A_2) \ge 2^{-O_{m,\eps}(\log^2(2/\alpha))}.\]  
Now, we invoke \Cref{thm:almost_periodicity} with $U_1=U_2 \defeq A'\subseteq B'$ and $A_1, A_2\subseteq B''$, and obtain $B'''$ and $x_0 \in G$ with $(A+x_0)\cap B'''\geq (1+\delta)\alpha$, such that
\begin{align*}
\rank(B{'''})&\leq \rank(B')+O_{m,\eps}(\log^2(2\alpha^{-1})\log(2\alpha_1^{-1})\log(2\alpha_2^{-1}))\\
&\leq \rank(B')+O_{m,\eps}\left(\log^{6}(2\alpha^{-1})\right)
\end{align*}
and \[|B'''|\geq 2^{-O_{m,\eps}(r)} |B''|,\]
where
\[C = \log(2/\alpha)+\log(\rank(B'))+\log\log(2/\alpha_1)+\log\log(2/\alpha_2)=O(\log(2/\alpha))\]
and
\begin{align*}
r&\leq  
C\cdot
(
\rank(B')+\log^2(2\alpha^{-1})\log(2\alpha_1^{-1})\log(2\alpha_2^{-1})
)\\
&\leq
C\cdot
\left(\rank(B') + \log^{6}(2\alpha^{-1})\right)\\
&\leq O\left( \log(2\alpha^{-1})d + \log^{7}\left(2\alpha^{-1}\right)\right).
\end{align*}

By an iterative density increment argument, similar to \Cref{theorem:AlgCountingGeneral}, we obtain a Bohr set $B^*$ and $x_* \in G$ such that $A_*\coloneqq (A+x_*)\cap B^*$ is $(\delta',d',r)$-algebraically spread in $B^*$ where $d' =  O_m(\log^{6}(2/\alpha))$ and $r=O_m(\log^{8}(2/\alpha))$,  $\delta' = \delta(\eps,m)$, and  $\alpha_*\coloneqq \mu_{B^*}(A_*)\geq\alpha$. We have 
     \[\rank(B^*)\leq O_{m,\eps}(\log^{7}(2/\alpha))\] 
    and \[\mu(B^*)\geq \exp\left(- O_{m,\eps}(\log^{9}(2/\alpha))\right).\]
  A similar argument as in the proof of \Cref{thm:L-freeAbelian}~(i) concludes the proof.
\qed

\section{Concluding remarks: resilience to spreadness}\label{section:concluding}  
In this section, we discuss negative results and argue that, essentially, the approach based on the graph and algebraic spreadness does not apply to most systems of linear forms.  For simplicity, we focus on the case of the group $\mathbb{F}^n$ where $\mathbb{F} =\mathbb{F}_q$ is a finite field, and $q$ is a fixed prime. The following notion of pseudo-randomness, based on graph spreadness, is implicit in the proof of \cref{thm:L-freeAbelian}.  

\begin{definition}[Spreadness~\cite{KM23,KLM}] 
\label{def:alg_spread}
Let $r \ge 1$ and $\delta >0$. A set $A \subseteq \F^n$ of density $\alpha$ is called
$(\delta,r)$-\emph{combinatorially spread} if, for all subsets $S$ and $T$ of density at least $2^{-r}$, we have
\[\Ex_{\substack{\bfx \sim S \\ \bfy \sim T}} [A(\bfx+\bfy)] \le (1+\delta) \alpha.\]
\end{definition}

In the literature, two main notions are commonly used to characterize the complexity of systems of linear forms.  The first notion is called the \emph{Cauchy-Schwarz complexity} (\emph{CS-complexity} for short), defined by Green and  Tao~\cite{green2010linear} in their study of solutions of linear equations in the primes. The second notion, defined by Gowers and Wolf~\cite{MR2773103,MR2578471},  is called the \emph{true complexity}. In both definitions, the complexity of a system is a positive integer. The following inclusions follow from the definitions of these complexity measures.  
\[ \set{\cL \ : \ \text{$\cL$ is binary} } \subseteq \set{\cL \ : \ \text{$\cL$ has CS-complexity $1$}}   \subseteq \set{\cL \ : \ \text{$\cL$ has true complexity $1$}}.   \]

Our sparse graph counting lemma shows that for binary systems of linear forms, combinatorial spreadness implies  $t_{\cL}(A) \approx   \alpha^m$. On the other hand, we will show in \cref{thm:true_complexity} that if the true complexity of a linear system $\cL$ is larger than $1$, then there are sets $A$ with density $\alpha$ that are combinatorially spread with extremely good parameters, and yet $t_{\cL}(A) > (1+\Omega(1)) \alpha^m$.

\paragraph{True complexity and CS-complexity.} Gowers, in his proof of Szemer\'edi's theorem, introduced a hierarchy of increasingly stronger notions of pseudo-randomness based on the so-called Gowers uniformity norms $\norm{\cdot}_{U^{k}}$ for $k \ge 1$. He used iterated applications of the classical Cauchy--Schwarz inequality to prove    
\[ \left|\Ex_{\bfx,\bfy}\left[ A(\bfx)A(\bfx+\bfy)\cdots A(\bfx+(k-1)\bfy)\right] -\alpha^k\right| \le k \norm{A-\alpha}_{U^{k-1}}, \]
showing that the density of $k$-progressions in $A$ is controlled by the pseudo-randomness condition $\norm{A-\alpha}_{U^{k-1}}=o(1)$.

Green and Tao~\cite{green2010linear} determined the most general class of systems of linear forms that can be handled by Gowers' iterated Cauchy--Schwarz argument. 
\begin{definition}[The Cauchy--Schwarz Complexity]\label{def:cscomplexity}
Let $\cL = \{L_1,\ldots, L_m\}$ be a system of linear forms. The \emph{Cauchy--Schwarz complexity} of $\cL$ is the minimal $k$ such that the following holds. For every
$1 \le i \le m$, we can partition  $\{L_j\}_{j \in [m] \setminus \{i\}}$ into $k+1$ subsets, such that $L_i$ does not belong to the
linear span of any of these subsets.
\end{definition}
In particular, Green and Tao~\cite{green2010linear} showed that if the Cauchy--Schwarz complexity of $\cL$ is $k$, then 
\begin{equation}
\label{eq:CS-complexity}
\left|t_{\cL}(A) -\alpha^m\right| \le m \norm{A-\alpha}_{U^{k+1}}.
\end{equation}

\begin{example}
The system of linear forms $(x,x+y,\ldots,x+(k-1)y)$, which represents $k$-term arithmetic progressions, has CS-complexity $k-2$. It is also straightforward to verify that binary systems of linear forms have CS-complexity $1$. 
\end{example}

Later, in a series of articles~\cite{MR2773103,MR2578471}, Gowers and Wolf initiated a systematic study of classifying the systems of linear forms that are controlled by the $k$-th Gowers uniformity norm. They defined the \emph{true complexity} of a system of linear forms $L=(L_1,\ldots,L_m)$ as the smallest $k$ such that the pseudo-randomness condition of $\norm{A-\alpha}_{U^{k+1}} =o(1)$  implies $|t_{\cL}(A)-\alpha^m|=o(1)$. 

By \eqref{eq:CS-complexity}, the true complexity is at most the CS-complexity. 
Gowers and Wolf~\cite{MR2773103} fully characterized the true complexity of systems of linear forms, provided that the field size is not too small.  Let $L=\lambda_1 x_1+\cdots+\lambda_d x_d$ be a linear form in $d$ variables. We identify $L \equiv (\lambda_1,\ldots,\lambda_d) \in \F^d$, and define the $k$-th tensor power of $L$ by
\[ L^k = \left(\prod_{j=1}^k \lambda_{i_j} : i_1,\ldots,i_k \in [d] \right) \in \F^{d^k}. \]

\begin{theorem}[\cite{MR2773103}]
Let $\cL=(L_1,\ldots, L_m)$ be a system of linear forms of CS-complexity at most $|\F|$. The \emph{true complexity} of $\cL$ is the least integer $k\ge 1$ such that the forms $L_1^{k+1},\ldots,L_m^{k+1}$ are linearly independent.
\end{theorem}

It is known that the true complexity of the linear system identifying $4$-progressions is 2~\cite{MR2578471,MR2773103}. Moreover, it is well-known that there is a set $A$ of density $\alpha$ such that $\norm{A-\alpha}_{U^{2}} =o(1)$ (which implies combinatorial spreadness) but the count of $4$-progressions in $A$ is far from what is expected from a random set of density $\alpha$~\cite{MR1844079}. Here, we generalize this result.
We essentially show that if the true complexity of a linear system $\cL$ is larger than $1$, then there are sets $A$ with density $\alpha$ that are combinatorially spread with extremely good parameters, and yet $t_{\cL}(A) \not\approx \alpha^m$.
 
\begin{restatable}{thm}{TrueComplexity}
\label{thm:true_complexity}
Let $\F$ be a finite field of prime order.
Suppose that $\cL=(L_1,\ldots, L_m)$ is a system of linear forms over $d$ variables such that $L_1^2,\ldots,L_m^2$ are linearly dependent. There exists a set $A \subseteq \F^n$ with density $\alpha \ge \frac{1}{3}$ that is $(e^{-\Omega(n)},\Omega(n))$-combinatorially spread,
yet 
\[t_{\cL}(A)> (1+\Omega(1))\alpha^m.  \]
\end{restatable} 
\noindent As the proof makes clear, the constant $\tfrac13$ can be replaced by any constant strictly smaller than~$1$.
 
\begin{proof}
Let $\e_\F\colon \F \to \mathbb{C}$ be an arbitrary non-trivial character of the additive group of $\F$. Concretely, if we identify $\F$ with $\{0,\dots,q-1\}$, then we can take $\e_\F(x) =e^{\frac{2 \pi i}{|\F|}x}$. 
 For a given $\rho = (\rho_1,\ldots,\rho_m)\in \mathbb{F}^m$, let
\[
\Delta_\rho \defeq 
\left|\Ex_{\substack{\bfx \sim \F^d \\ \bfy \sim \F^d } } \e_{\F}\left(\sum_{i=1}^m  \rho_i  L_i(\bfx)L_i(\bfy)   \right) \right|.
\]
\paragraph{Estimating $\Delta_\rho$.} Let $\Gamma$ be the set of all $\rho=(\rho_1,\ldots,\rho_m) \in \F^m$ such that 
\[
\sum_{i=1}^m \rho_i L_i^2 \equiv0. 
\]
Hence, if $\rho \in \Gamma$, we have $\Delta_\rho = 1$.

Next, we show that if  $\rho \in \mathbb{F}^m\backslash \Gamma$, then $\Delta_\rho \leq \frac{1}{\sqrt{|\F|}}$.
Let $L_i (x) = \sum_a \lambda i_a x_a$. Then
\begin{align*}
  \sum_{i=1}^m \rho_i L_i(x)L_i(y) =&
 \sum_{i=1}^m \rho_i \sum_{a,b=1}^d \lambda_{ia} \lambda_{ib} x_a y_b \\
 =&\sum_{a,b=1}^d  \left(\sum_{i=1}^m \rho_i \lambda_{ia} \lambda_{ib}\right) x_a y_b.
\end{align*}
For $\rho \in \mathbb{F}^m\backslash \Gamma$, since
\[
  0 \not\equiv\sum_{i=1}^m \rho_i L_i^2 = \left(\sum_{i=1}^m \rho_i \lambda_{ia}\lambda_{ib} : a,b \in [d] \right),
\]
there are $a,b$ such that  
\[\alpha \defeq  \sum_{i=1}^m \rho_i \lambda_{ia} \lambda_{ib} \not\neq 0.\]
Therefore, for every fixed assignment to the variables $\{x_i : i\in [d]\backslash\{a\}\} \cup \{y_i : i\in [d]\backslash\{b\}\}$, there exist $\beta,\gamma,\psi \in \F$ such that 
\[\sum_{i=1}^m \rho_i L_i(x)L_i(y)  = \alpha x_ay_b + \beta x_a+\gamma y_b + \psi. \]
By the standard Gauss sum estimate~\cite[Lemma 3.1]{MR2773103},   we have
\[\left|\Ex_{\bfx_a,\bfx_b\in \mathbb{F}}\e_{\F}\Bigg(\alpha \bfx_a\bfx_b + \beta \bfx_a+\gamma \bfx_b+\psi \Bigg)\right|\leq \frac{1}{\sqrt{|\F|}}.
\]
Averaging over all assignments to $\{x_i : i\in [d]\backslash\{a\}\} \cup \{y_i : i\in [d]\backslash\{b\}\}$, we obtain
\begin{align*}
\Delta_\rho \leq \frac{1}{\sqrt{|\F|}}.
\end{align*}

\paragraph{Constructing a spread function $f$.}
Define $f\colon \F^n\to [0,1]$ as 
\[f(x) \defeq  G(Q(x)),\]
where 
 $G\colon \F \to [0,1]$ is defined as
 \[G(z) \defeq \frac{1}{2}+\sum_{0 \neq \gamma \in \F}\frac{1}{2|\F|}\e_\F(\gamma z),\]
 and   
 $Q\colon \F^n \to \F$ as
$Q(y) =y_1y_2+y_3y_4+\cdots+y_{n-1}y_{n}$,
where $n$ is some even number.
The distribution of $Q(\bfx)$ approaches the uniform distribution as $n \to \infty$, and in particular, $\Ex[f]$ approaches $\Ex[G]= \frac{1}{2}$. 
Using the Fourier transform of $G$, we have 
\[f(x) = \sum_{\gamma \in \F} \widehat{G}(\gamma) \e_{\F}(\gamma Q(x)).\]
For $x_1,\ldots,x_d \in \F^n$, we have 
\[\prod_{i=1}^m f\left(L_i(x_1,\ldots,x_d)\right) = \prod_{i=1}^m f\left(\sum_{j=1}^d \lambda_{ij}  x_j\right).\]
By substituting the Fourier transform of $G$ and expanding it, the above equals
\[
\sum_{\rho  \in \F^m} \widehat{G}(\rho_1)\cdots \widehat{G}(\rho_m) \e_{\F}\left(\sum_{i=1}^m  \rho_i  Q \left(\sum_{j=1}^d \lambda_{ij}  x_j\right)   \right).\]
By the definition of $Q$, we have  
\begin{align*}
\Ex_{\bfx_1,\ldots,\bfx_d \sim \F^n} \e_{\F}\left(\sum_{i=1}^m  \rho_{i}  Q \left(\sum_{j=1}^d \lambda_{ij}  \bfx_j\right)   \right) 
&= \left(\Ex_{\substack{\bfx_1,\ldots,\bfx_d \sim \F \\ \bfy_1,\ldots\bfy_d \sim \F}} \e_{\F}\left(\sum_{i=1}^m  \rho_{i}  \left(\sum_{j=1}^d \lambda_{ij}  \bfx_j\right)\left(\sum_{j=1}^d \lambda_{ij}  \bfy_j\right)   \right) \right)^{\frac n 2} \\
&= \Delta_\rho^{n/2}= 
\begin{cases}
1 & \text{if } \rho \in \Gamma, \\
\leq \left(\frac{1}{\sqrt{|\F|}}\right)^{n/2} & \text{if } \rho \not\in \Gamma,
\end{cases}
\end{align*}
which shows that
\[t_\cL(f) \defeq  \Ex \left[ \prod_{i=1}^m f(L_i(\bfx_1,\ldots,\bfx_d)) \right] 
  =\sum_{\rho \in \Gamma} \widehat{G}(\rho_1)\cdots \widehat{G}(\rho_m)  \pm O\left(|\F|^{-n/4}\right).\]
Note that $\widehat{G}(\gamma) \geq \frac{1}{2|\F|}$ for all $\gamma\in \F$. 
Moreover, since $L_1^2,\hdots L_m^2$ are linearly \emph{dependent}, $\Gamma$ contains at least one non-zero element. Therefore, the positivity of $\widehat{G}$ implies that for large enough $n$, 
\begin{equation}
\label{eq:large_f_t}
t_\cL(f) = \Ex[G]^m+  \sum_{\rho \in \Gamma\backslash\{0\}} \widehat{G}(\rho_1)\cdots \widehat{G}(\rho_m)  \pm O\left(|\F|^{-n/4}\right)= \Ex[f]^m+\Omega(1).
\end{equation}

Next, note that for every non-zero $\chi \in \F^n$, by the standard Gauss sum estimate~\cite[Lemma 3.1]{MR2773103},
\[|\widehat{f}(\chi)| = \left|\sum_{\gamma \in \F} \widehat{G}(\gamma) \Ex_{\bfx} \e_{\F}\left(\gamma Q(\bfx) - \chi (\bfx)\right) \right| \le e^{-\Omega(n)}.\]
Hence, denoting $\beta\defeq\Ex[f]=\widehat{f}(0)$, 
\[\max_{\chi \neq 0} |\widehat{f}(\chi)|=\norm{\widehat{f-\beta}}_\infty \le e^{-\Omega(n)},\]
which implies that for $S,T \subseteq \F^n$, we have 
\begin{align*}
\left|\Ex_{\bfx,\bfy} \1_S(\bfx)(f(\bfx+\bfy)-\beta)\1_T(\bfy)\right| 
&= \left|\sum_{\chi \neq 0} \widehat{\1_S}(\chi)\widehat{f}(-\chi)\widehat{\1_T}(\chi)\right| \le \left(\max_{\chi \neq 0} |\widehat{f}(\chi)|\right) \sum_{\chi}  |\widehat{\1_S}(\chi) \widehat{\1_T}(\chi)| \\
&\le  \norm{\widehat{f-\beta}}_\infty 
 \norm{\1_S}_{2}\norm{\1_T}_{2}\le e^{-\Omega(n)} \norm{\1_S}_{2}\norm{\1_T}_{2}= \frac{e^{-\Omega(n)}\sqrt{|S||T|}}{|\F|^n}.
\end{align*}
For an appropriate constant $c>0$,  if $|S|,|T| \ge |\F|^{n-cn}$, then 
\[\left| \beta -\Ex_{\substack{\bfx \sim S \\ \bfy \sim T}} f(\bfx+\bfy) \right| \le  \frac{|\F|^{2n}}{|S||T|} \times  \frac{e^{-\Omega(n)}\sqrt{|S||T|}}{|\F|^n}=\frac{e^{-\Omega(n)} |\F|^n}{\sqrt{|S||T|}}=e^{-\Omega(n)}.
\]
\paragraph{Converting $f$ to a set $A\subset \F^n$.}
To convert $f$ to a subset $A \subseteq \F^n$, we pick a random subset $\bfA \subseteq \F^n$ by including every element $x \in \bfA$ with probability $f(x)$. Note that 
\begin{equation}
\label{eq:spread_f_to_A}
|\Ex_{\bfx,\bfy} \1_S(\bfx)(f(\bfx+\bfy)-\1_\bfA(\bfx+\bfy))\1_T(\bfy) |\le  \norm{\1_S}_{2}\norm{\1_T}_{2} \norm{\widehat{f-\1_\bfA}}_\infty = \frac{|\F|^{2n}}{|S||T|} \norm{\widehat{f-\1_\bfA}}_\infty.
\end{equation}
For every $\chi \in \F^n$, we have 
\[\widehat{\1_{\bfA}}(\chi)=\frac{1}{|\F^n|} \sum_{x \in \F^n} \1_\bfA(x) \e_\F(x \cdot \chi), \]
and 
\[\Ex_{\bfA}\widehat{\1_{\bfA}}(\chi)=\widehat{f}(\chi). \]
Set $t \defeq |\F|^{2n/3}$.  By Hoeffding's inequality,  
\[\Pr_{\bfA}\left[\left|\widehat{\1_{\bfA}}(\chi)-\widehat{f}(\chi)\right| \ge \frac{t}{|\F|^n}\right] \le 2 \exp(-\frac{2 t^2}{4 |\F|^n}), \]
and so, by the union bound, 
\[\Pr_{\bfA}\left[\norm{\widehat{f-\1_{\bfA}}}_\infty  \ge \frac{t}{|\F|^n}\right] \le 2 \exp(-\frac{2 t^2}{4 |\F|^n}) |\F|^n=o(1).\]
In particular, by \eqref{eq:spread_f_to_A} and the choice of $t$, with probability $1-o(1)$, the set $\bfA$ has the following two properties: 
\begin{itemize}
\item[(i)] $\Ex[\1_\bfA] = \Ex[f] \pm  e^{-\Omega(n)}$, and in particular $\Ex[\1_\bfA] \ge 1/3$.  
\item[(ii)] For every $|S|,|T| \ge |\F|^{n - cn}$,  we have
\[\left| \Ex[f] -\Ex_{\substack{\bfx \sim S \\ \bfy \sim T}} A(\bfx+\bfy) \right| \le e^{-\Omega(n)}.\]
\end{itemize}
Next, we consider $|t_{\cL}(\bfA)-t_{\cL}(f)|$.  First, note  that by the definition of $\bfA$, we have  
\begin{align*}
|\Ex_\bfA \left[t_{\cL}(\bfA)-t_{\cL}(f)\right]|&=\left| \Ex_{\bfx,\bfA}\left[ \prod_{i=1}^m\1_{\bfA}(L_i(\bfx))\right] - \Ex_{\bfx}\left[ \prod_{i=1}^m f(L_i(\bfx))\right]\right|.
\end{align*}
Observe that if, for a given $x$, all $L_i(x)$'s are distinct, then 
\[
\Ex_{\mathbf{A}}\left[\prod_{i=1}^m\1_{\bfA}(L_i(x))\right] - \prod_{i=1}^m f(L_i(x)) = 0. 
\]
Hence, 
\begin{align*}
|\Ex_\bfA \left[t_{\cL}(\bfA)-t_{\cL}(f)\right]|&=\left| \Ex_{\bfx,\bfA}\left[ \prod_{i=1}^m\1_{\bfA}(L_i(\bfx))\right] - \Ex_{\bfx}\left[ \prod_{i=1}^m f(L_i(\bfx))\right]\right| \\
&\le \Pr_{\bfx}\left[\exists i\neq j \text{ s.t. } L_i(\bfx)=L_j(\bfx)\right] \le  m^2 |\F|^{-n}.
\end{align*}
Moreover,  
\begin{align*}
\Ex_{\bfA} \left|t_{\cL}(\bfA) - t_{\cL}(f)\right|^2 &= 
\Ex_{\bfA} \left[t_{\cL}(\bfA)^2-t_{\cL}(\bfA)t_{\cL}(f)\right] - t_\cL(f) \Ex_{\bfA} \left[t_{\cL}(\bfA) - t_{\cL}(f)\right]
\\
&\le \left|\Ex_{\bfA}  \Ex_{\bfx,\bfy}\left[ \prod_{i=1}^m\1_{\bfA}(L_i(\bfx))\1_{\bfA}(L_i(\bfy)) - \prod_{i=1}^m \1_{\bfA}(L_i(\bfx))f(L_i(\bfy))\right]\right| +  m^2 |\F|^{-n}. 
\end{align*}
Note that for $x,y\in (\F^n)^d$, if $L_1(x),\ldots,L_m(x),L_1(y),\ldots,L_m(y)$ are distinct, then 
\[
\Ex_{\bfA} \left[ \prod_{i=1}^m\1_{\bfA}(L_i(x))\1_{\bfA}(L_i(y)) - \prod_{i=1}^m \1_\bfA(L_i(x))f(L_i(y))\right] = 0.
\]
 Hence, 
\begin{align*}  
\Ex_{\bfA} \left|t_{\cL}(\bfA) - t_{\cL}(f)\right|^2  
&\le \Pr_{\bfx, \bfy}[L_1(\bfx),\ldots,L_m(\bfx),L_1(\bfy),\ldots,L_m(\bfy) \text{ are not distinct}] +  m^2 |\F|^{-n} \\
& \le 4m^2 |\F|^{-n}.
\end{align*}
Therefore, Markov's inequality implies that
\[\Pr_{\bfA}[\left|t_{\cL}(\bfA) - t_{\cL}(f)\right| \ge |\F|^{-n/3}]=o(1),\]
which, in combination with \eqref{eq:large_f_t}, shows that, with probability $1-o(1)$, the set $\bfA$ satisfies 
\begin{itemize}
\item[(iii)] $t_{\cL}(\bfA) = \Ex[f]^m + \Omega(1)$. 
\end{itemize}
Since with high probability, the random set $\bfA$ satisfies (i), (ii), (iii), there exists a set satisfying these three properties. 
\end{proof}

\bibliographystyle{alpha}  
\bibliography{ref}

\end{document}

%% file: myheader.tex
\usepackage{amsmath,amsthm, amssymb,amscd,amstext,amsfonts}
\usepackage[utf8]{inputenc}
\usepackage{mathtools}
\usepackage{mathrsfs}
\usepackage{thmtools}
\usepackage{latexsym}
\usepackage{verbatim}
\usepackage{framed}
\usepackage{graphicx}
\usepackage{stmaryrd}
\usepackage{enumitem}
\usepackage{fullpage}
\usepackage{bm}
\usepackage{url}
\usepackage{physics}
\usepackage{dsfont}
\usepackage{todonotes}
 
\usepackage{mathtools}

\usepackage{color}
\usepackage[hidelinks]{hyperref}

\usepackage[nameinlink, noabbrev, capitalize]{cleveref}

\usepackage{crossreftools}
\theoremstyle{plain}

\newtheorem{thm}{Theorem}[section]
\newtheorem{theorem}[thm]{Theorem}

\newtheorem{corollary}[thm]{Corollary}
\newtheorem{lemma}[thm]{Lemma}
\newtheorem{proposition}[thm]{Proposition}

\newtheorem{example}[thm]{Example}

\newtheorem{claim}[thm]{Claim}
\newtheorem*{claim*}{Claim}

\theoremstyle{definition}
\newtheorem{definition}{Definition}

\theoremstyle{remark}
\newtheorem{remark}{Remark}

\renewcommand{\norm}[1]{\|#1\|}                     % norm
\newcommand{\set}[1]{\{#1\}}
\newcommand{\inp}[2]{{\left\langle #1,#2 \right\rangle}}            % inner product
\DeclareMathOperator{\conv}{conv}                                             % convex hull
\def\rank{{\mathrm{rk}}}                               % rank
\DeclarePairedDelimiter\ceil{\lceil}{\rceil}             % integer part
\DeclareMathOperator{\Ex}{\mathbb{E}}           % expected value 
\DeclareMathOperator{\Bohr}{Bohr}

\newcommand{\bfy}{\mathbf y}

\newcommand{\bfx}{\mathbf x} 
\newcommand{\bfz}{\mathbf z} 
\newcommand{\bfb}{\mathbf b} 
\newcommand{\bfs}{\mathbf s} 
\newcommand{\bft}{\mathbf t}

\newcommand{\bfA}{\mathbf A} 

\newcommand{\bfw}{\mathbf w}

\newcommand{\N}{\mathbb{N}}

\newcommand{\R}{\mathbb{R}}
\newcommand{\F}{\mathbb{F}}
\newcommand{\Z}{\mathbb{Z}}

\newcommand{\cB}{\mathcal B}

\newcommand{\cL}{\mathcal L}

\newcommand{\cX}{\mathcal X}
\newcommand{\cY}{\mathcal Y}

\newcommand{\defeq}{\coloneqq}

\newcommand{\eps}{\varepsilon}

\def\1{\mathbf{1}} 
\def\0{\mathbf{0}}

\newcommand{\e}{\mathbf e}